\documentclass[11pt,pdflatex,sn-mathphys-num]{sn-jnl}
\usepackage{amsfonts,amssymb,amsmath}
\usepackage{graphicx}
\usepackage{epstopdf}
\usepackage{algorithmic}
\usepackage{booktabs}
\usepackage{wrapfig}
\usepackage{tikz,pgfplots}
\usepackage{pgfplotstable}
\usepackage[inline]{enumitem}
\usepgfplotslibrary{fillbetween}
\usetikzlibrary{pgfplots.groupplots}
\pgfplotsset{compat=1.18}
\usetikzlibrary{arrows,shapes,trees,calc}
\usetikzlibrary{decorations,decorations.markings}

\pgfkeys{/tikz/.cd,
    contour distance/.store in=\ContourDistance,
    contour distance=-10pt, %
    contour step/.store in=\ContourStep,
    contour step=1pt,
}

\pgfdeclaredecoration{closed contour}{initial}
{%
\state{initial}[width=\ContourStep,next state=cont] {
    \pgfmoveto{\pgfpoint{\ContourStep}{\ContourDistance}}
    \pgfcoordinate{first}{\pgfpoint{\ContourStep}{\ContourDistance}}
    \pgfpathlineto{\pgfpoint{0.3\pgflinewidth}{\ContourDistance}}
    \pgfcoordinate{lastup}{\pgfpoint{1pt}{\ContourDistance}}
    
  }
  \state{cont}[width=\ContourStep]{
     \pgfmoveto{\pgfpointanchor{lastup}{center}}
     \pgfpathlineto{\pgfpoint{\ContourStep}{\ContourDistance}}
     \pgfcoordinate{lastup}{\pgfpoint{\ContourStep}{\ContourDistance}}
  }
  \state{final}[width=\ContourStep]
  { %
    \pgfmoveto{\pgfpointanchor{lastup}{center}}
    \pgfpathlineto{\pgfpointanchor{first}{center}}
  }
}

\makeatletter
\newtheoremstyle{thmstyleone}%
{10pt plus2pt minus1pt}%
{10pt plus2pt minus1pt}%
{\itshape}%
{0pt}%
{\bfseries}%
{}%
{.5em}%
{\thmname{#1}\thmnumber{\@ifnotempty{#1}{ }\@upn{#2}}%
  \thmnote{ {\the\thm@notefont(#3)}}}%
\newtheoremstyle{thmstyletwo}%
{10pt plus2pt minus1pt}%
{10pt plus2pt minus1pt}%
{\normalfont}%
{0pt}%
{\itshape}%
{}%
{.5em}%
{\thmname{#1}\thmnumber{\@ifnotempty{#1}{ }{#2}}%
  \thmnote{ {\the\thm@notefont(#3)}}}%
\newtheoremstyle{thmstylethree}%
{10pt plus2pt minus1pt}%
{10pt plus2pt minus1pt}%
{\normalfont}%
{0pt}%
{\bfseries}%
{}%
{.5em}%
{\thmname{#1}\thmnumber{\@ifnotempty{#1}{ }\@upn{#2}}%
  \thmnote{ {\the\thm@notefont(#3)}}}%
\AtBeginDocument{%
  \renewenvironment{spiproof}[1][\proofname]{\par\removelastskip%
    \pushQED{\qed}%
    \normalfont \topsep7.5\p@\@plus7.5\p@\relax%
    \trivlist%
    \item[\hskip\labelsep\itshape #1\@addpunct{}]\ignorespaces%
  }{%
    \popQED\endtrivlist\@endpefalse%
  }%
}%
\makeatother

\makeatletter
\g@addto@macro\abstractfont{\mathversion{normal}}
\g@addto@macro\keywordfont{\mathversion{normal}}
\makeatother

\theoremstyle{thmstyleone}%
\newtheorem{theorem}{Theorem}%
\newtheorem{lemma}[theorem]{Lemma}%
\newtheorem{corollary}[theorem]{Corollary}%

\theoremstyle{thmstyletwo}%

\theoremstyle{thmstylethree}%

\newcommand{\bs}[1]{\ensuremath{\boldsymbol{#1}}}

\usepackage{cleveref}

\usepackage{amsopn}

\DeclareMathOperator{\sym}{sym}

\begin{document}

\title[]
{{\Large Transform before linearizing: robust Newton methods for singular $p$-Laplace and $p$-Stokes equations}}

\author[1,2]{\fnm{Gonzalo G.} \sur{De Diego}}\email{gg2924@nyu.edu}

\author[1]{\fnm{Fortino} \sur{Garcia}}\email{foga8952@colorado.edu}

\author[1]{\fnm{Georg} \sur{Stadler}}\email{stadler@cims.nyu.edu}

\author[1]{\fnm{Kewei} \sur{Wang}}\email{kw3645@nyu.edu}

\affil[1]{\orgdiv{Courant Institute School of Mathematics, Computing and Data Science}, \orgname{New York University},
\city{New York},
\state{NY}, \country{USA}}

\affil[2]{\orgdiv{ETSIT}, \orgname{Technical University of Madrid (UPM)}, \city{Madrid}, \country{Spain}}

\abstract{For $1<p<2$, the $p$-Laplace equation and its $p$-Stokes
generalization are difficult to solve numerically. Newton's method converges rapidly only close to the solution, with iteration counts that grow under mesh refinement and deteriorate as $p\to1$. The more robust Picard iteration converges only linearly.
Rather than globalizing or preconditioning Newton's method, we modify the system to which it is applied. We \emph{lift} the equation by introducing the flux $|\nabla u|^{p-2}\nabla u$
as an auxiliary (or ``lifting'') variable, apply a nonlinear \emph{transformation} to
the resulting constitutive relation, \emph{linearize}, and \emph{eliminate} the auxiliary variable by static condensation. Lifting alone leaves the linearization
unchanged; it is the preceding transformation that yields the new method. The elimination is algebraic and pointwise at the quadrature points, so the flux variable is never discretized, no inf-sup condition or indefinite system arises, and the cost per iteration is that of a standard Newton step. Together with a pointwise feasibility bound on the lifting variable that keeps the diffusion tensor uniformly positive definite, this yields an iteration that we prove, in finite dimensions and for the $p$-Laplace equation, to converge globally and locally at a quadratic rate; a one-dimensional model problem explains why the lagged flux variable removes the zig-zag behavior of standard Newton for $p$ close to one. Firedrake-based experiments for $p$-Laplace problems in two and three dimensions and for stationary and time-dependent $p$-Stokes flows show iteration counts largely insensitive to $p$ and to mesh refinement, and up to an order of magnitude fewer iterations than standard Newton for $p$ close to one.}



\maketitle

\section{Introduction}
\label{sec:lifting}
For $p\in(1,\infty)$, the $p$-Laplace equation 
\begin{equation}
\label{eq:p-laplace-intro-toplevel}
-\nabla\cdot\!\bigl(|\nabla u|^{p-2}\nabla u\bigr) = f
\quad\text{in } \Omega,
\end{equation}
defined over a domain $\Omega\subset \mathbb R^d$ is the prototypical quasilinear elliptic problem. It arises as the Euler--Lagrange equation of the $p$-Dirichlet energy $J(u)=\tfrac{1}{p}\!\int_{\Omega}\!|\nabla u|^{p}\,\mathrm{d}\bs x
   - \!\int_{\Omega}\! f\,u\,\mathrm{d}\bs x$.
For $p=2$, it reduces to the linear Poisson problem; for any other
exponent, it is genuinely nonlinear.
Operators similar to the $p$-Laplacian
occur in a wide range of applications:
power-law shear-thinning or thickening models for non-Newtonian fluids, ice-sheet dynamics under Glen's flow corresponding to $p=4/3$ \cite{CuffeyPaterson10,glaciology-Newton}, problems in geodynamics \cite{TurcotteSchubert02}, 
and motion by curvature in the limit $p\to\infty$ \cite{KohnSerfaty06,Oberman13}. 

Despite its seemingly simple form, \eqref{eq:p-laplace-intro-toplevel}
is notoriously difficult to solve numerically. Two features account for much of this difficulty. First, the diffusion coefficient $|\nabla u|^{p-2}$ concentrates the nonlinearity in a quantity that becomes \emph{singular} for $1<p<2$, i.e., it diverges at points where $\nabla u=0$, and \emph{degenerate} for $p>2$, where the linearized operator loses ellipticity.
Second, solutions are not classically smooth: weak solutions are only known to be $C^{1,\alpha}$ for some $\alpha\in(0,1)$, and regularity properties deteriorate as $p$ moves away from $2$. Both features have consequences at the numerical level, where the most commonly used iterative schemes to solve \eqref{eq:p-laplace-intro-toplevel} are \emph{Newton} and \emph{Picard} methods.

Newton's method linearizes the nonlinear operator on the left hand side of \eqref{eq:p-laplace-intro-toplevel}
at a current iterate $u_k$. Writing the unit gradient direction $\bs{w}_k := \nabla u_k/|\nabla u_k|$ (and, for now, assuming that the gradient does not vanish), the Newton increment $\hat u$ solves
\begin{equation}
\label{eq:newton-intro}
-\nabla\cdot\left(|\nabla u_k|^{p-2}
\bigl(I + (p-2)\,\bs{w}_k\bs{w}_k^\top\bigr)\nabla\hat u\right)
= \nabla\cdot\bigl(|\nabla u_k|^{p-2}\nabla u_k\bigr) + f
\quad\text{in }\Omega,
\end{equation}
followed by the update $u_{k+1} := u_k + \alpha\hat u$ with step length $\alpha\in(0,1]$.

Instead, Picard freezes the diffusion coefficient $|\nabla u|^{p-2}$ at the current iterate $u_k$ and computes $\hat u$ by solving the linear problem
\begin{equation}
\label{eq:picard-intro}
-\nabla\cdot\left(|\nabla u_k|^{p-2}\nabla \hat u\right) = \nabla\cdot\bigl(|\nabla u_k|^{p-2}\nabla u_k\bigr) + f
\quad\text{in }\Omega,
\end{equation}
followed again by $u_{k+1} := u_k + \alpha\hat u$ with $\alpha\in (0,1]$.
A common, mathematically equivalent variant of \eqref{eq:picard-intro} is to write it directly in terms of $u_{k+1}$ under the assumption $\alpha=1$ (which is the typical choice for Picard).
Comparing \eqref{eq:newton-intro} and \eqref{eq:picard-intro}, it can be seen that the only difference is the anisotropic term $(p-2)\,\bs{w}_k\bs{w}_k^\top$, which can cause challenges for robust convergence.

The main challenges for Newton solvers based on \eqref{eq:newton-intro} (or an appropriate regularized form) are their \emph{small basin of fast convergence} and the resulting \emph{mesh-dependent iteration counts}. Newton's method converges quadratically only from initial iterates very close to the solution, with the radius of the guaranteed-convergence ball shrinking rapidly as the level of nonlinearity increases \cite{Deuflhard04,schmidt-pStokes-global}. This occurs when $p$ deviates from $2$, especially in the regime $1<p< 2$, which is the primary case we consider.
Damped, line-search-globalized, or semismooth Newton iterations are therefore essential in practice, but they typically require a large number of short steps before the iterate enters the quadratic regime---if it does at all. 
As a consequence, Newton solvers typically exhibit iteration
counts that grow under mesh refinement and depend severely on $p$
\cite{HuangLiLiu07, Loisel20, HeinleinKlawonnLanser22}.

A common alternative to Newton's method is the Picard method \eqref{eq:picard-intro}, which comes with its own drawbacks. It is more robust
than Newton globally, but it converges only linearly, with a contraction
factor that worsens as $p\to 1$.  Continuation strategies in $p$
\cite{HuangLiLiu07} alleviate this issue at the price of solving a sequence
of auxiliary problems.

\subsection{Approach and related literature}
In this paper, we propose the following strategy for singular $p$-Laplace and $p$-Stokes problems, i.e., $1<p<2$:  Rather than modifying Newton's or Picard's method by means of a preconditioner or continuation, we modify the system on which Newton's method operates: we
(1) \emph{lift} the original nonlinear equation to a
higher-dimensional space by introducing an auxiliary variable (here, the dual variable) that exposes the nonlinearity through a pointwise algebraic identity; (2) perform a nonlinear \emph{transformation} of the lifted system that tames the offending nonlinearity; (3) \emph{linearize} to derive Newton's method in the lifted space; and (4) \emph{eliminate} the auxiliary variable from the resulting linear system by static condensation.  The dual variable from the underlying convex-energy formulation provides a natural choice of lifting variable. For the $p$-Laplace equation \eqref{eq:p-laplace-intro-toplevel}, this is the flux $\bs{\lambda}=|\nabla u|^{p-2}\nabla u$, and for our generalization to $p$-Stokes problems, it is the deviatoric stress tensor.  Because the auxiliary unknown is eliminated before solving, the 
cost per iteration matches that of the standard Newton (or even Picard) method, and existing Newton solvers only require minor adjustments. As we will show, our approach exhibits a significantly larger region of rapid convergence,
with iteration counts that are largely unaffected by $p$ or mesh refinement.

Our approach is related to several existing lines
of research.  The lifted Newton method proposed in
\cite{AlbersmeyerDiehl10}, originally motivated by
direct multiple-shooting methods for optimal control
\cite{BockPlitt84},
likewise introduces auxiliary variables to expose nonlinearity, but does so
primarily in the context of time-dependent problems.  Nonlinear
elimination~\cite{lanzkron1996analysis} and composed nonlinear
solvers~\cite{brune2015composing} follow the philosophy of
restructuring the nonlinear system rather than its linearization.
Nonlinear preconditioning through domain decomposition
\cite{Gander17, DoleanGanderKherijiEtAl16, LiuKeyes16,
LiuHwangLuoEtAl22, HeinleinKlawonnLanser22} is complementary to our
approach in that it focuses iterations on locally difficult
subdomains; our lifting approach addresses the source of the difficulty directly.

In the context of nonlinear PDE solvers, related primal--dual and stress-lifting variants of Newton's method have been proposed for other nonlinear PDEs with singular constitutive laws.
For Bingham flow, whose constitutive law is singular below a yield stress and is therefore closely analogous to the $p$-Laplacian with $p<2$, \cite{DelosreyesGonzalez09} employs
a semismooth Newton method formulated using the dual stress. Close to our framework is the stress--velocity Newton scheme of \cite{RudiShihStadler20,Rudi18} for viscoplastic geodynamic Stokes flow, which treats the deviatoric stress as an independent variable prior to its elimination. The abstract form of that construction in \cite[Sec 5.3]{Rudi18} shares the steps of the strategy above and studies it for a term of $L^1$-type. It
explains the degeneracy of the standard linearization with its coefficient tending to an orthogonal projector with a vanishing eigenvalue along one direction, resulting in small step sizes.
Closely related, \cite{ShihMehlmannLoschEtAl22} proposes a primal--dual Newton--Krylov method for a viscous-plastic sea-ice model.
Instead of introducing a lifted stress variable, the authors of \cite{ZampiniBoffiDovletovMarkowich26} evolve a scalar lifted field that tracks $|\nabla u|^2$ via a gradient-flow evolution, recovering the classical Ka\v{c}anov iteration (which requires solving systems similar to \eqref{eq:picard-intro}), as a special case. They demonstrate robustness and iteration numbers competitive with Newton's method (which, as we argue here, is not a robust method, particularly for $p$ close to 1).
Retaining the stress as an independent unknown is also the premise of the implicit constitutive theory of non-Newtonian fluids~\cite{BulicekGwiazdaMalekEtAl12}, in which the constitutive law is a relation between the stress and the symmetric velocity gradient that need not be solvable for either variable; our lifting exploits the same structural freedom, but at the level of the linearization rather than the model.
Collectively, these studies indicate that reformulating the constitutive relation via an auxiliary variable is a promising approach for nonlinear PDE solvers. In this work, we develop a lifting-based method for the $p$-Laplace and $p$-Stokes systems and provide arguments supporting its superior convergence behavior.

Upon discretization and replacing the gradient operator with a general matrix $A$, the discrete $p$-Dirichlet energy associated with \eqref{eq:p-laplace-intro-toplevel} becomes an $\ell_p$-regression problem of the form $\min_{\bs x} \|A\bs x - \bs b\|_p^p$. 
This class of problems has been widely examined in optimization and theoretical computer science. Our lifting approach is most closely related to the interior-point framework of \cite{andersen2000primaldual}, which applies Newton's method to a coupled primal/dual feasibility system exhibiting a similar lift--linearize--condense pattern. In the limit $p \to 1$, primal--dual algorithms for total variation~\cite{chan1999nonlinear,HintermullerStadler06} introduce a similar dual variable combined with a feasibility modification, directly paralleling the feasibility projection we employ for the lifting variable. For the discrete $p$-Laplacian specifically, \cite{Loisel20} gives an algorithm that exploits the convex variational structure of the problem to compute a solution to high accuracy in polynomial time.

\subsection{Main contributions and limitations}
The main \emph{contributions} we make in this work are as follows: 
(1) Our approach is a simple modification of Newton's method that can be implemented in standard automated finite element software. It neither requires a mixed discretization nor the solution of indefinite systems.
(2) We show that the discretized iteration converges globally, with line search and pointwise diffusion tensor modification, for $1<p<2$. It is also locally quadratically convergent; 
an analysis of a one-dimensional model problem explains why the lagged lifting variable suppresses the zig-zag behavior of standard Newton in flat regions.
(3) Our publicly available implementations (in the finite element library Firedrake) show iteration counts largely insensitive to mesh refinement and to $p\to1$, in two and three dimensions for $p$-Laplace equations. Our generalizations to $p$-Stokes and time-dependent $p$-Navier--Stokes indicate an up to order-of-magnitude reduction in required linear solves compared to standard Newton for $p$ close to $1$.

The main \emph{limitations} are:
(1) The lifting approach is most advantageous for problems with strong nonlinearities and potentially singular behavior, such as the power-law constitutive relation in the singular regime $1<p<2$.
The proposed method is thus less broadly applicable than, e.g., the Picard--Newton methods of \cite{PollockRebholzTuXiao25}, which do not require such structure.
(2) Our convergence results are for the discretized and regularized $p$-Laplace equation only, and the constants in the analysis deteriorate as $p\to1$ at the same rate as for standard Newton. The robustness we observe numerically is thus not fully explained by the theory; we discuss it using intuitive arguments and illustrate it for a one-dimensional model problem.
(3) All linear systems in our experiments are solved with direct solvers. Although we demonstrate the method on challenging two- and three-dimensional problems with linear and quadratic finite elements and on the order of $10^5$ unknowns, we do not study the use of preconditioned iterative Krylov solvers.

\section{The $p$-Laplace equation}
\label{sec:p-laplacian}
In this section, we summarize the functional-analytic setting for the $p$-Laplace equation \eqref{eq:p-laplace-intro-toplevel} and introduce equivalent problem formulations that motivate different solution strategies.
We assume a bounded Lipschitz domain $\Omega \subset \mathbb{R}^d$ ($d\geq 1$), a source term
$f \in W^{-1,q}(\Omega)$ (with $1/p+1/q=1$), and given boundary
data $g \in W^{1/q,p}(\partial\Omega)$. Then the $p$-Laplace problem is
\begin{equation}
\label{eq:p-laplace}
-\nabla \cdot \bigl( |\nabla u|^{p-2}\, \nabla u \bigr) \;=\; f
\quad \text{in } \Omega, \qquad u = g \quad \text{on } \partial\Omega,
\end{equation}
where $p \in (1,\infty)$ controls the degree of
nonlinearity.
The natural setting is the Sobolev space
$W^{1,p}(\Omega)$: weak solutions of \eqref{eq:p-laplace} with homogeneous boundary data (i.e., $g=0$) are the unique minimizers of the strictly convex energy
\begin{equation}
J(u) \;:=\; \frac{1}{p}\int_\Omega |\nabla u|^{p}\,\mathrm{d}\bs x
       \;-\; \int_\Omega f\, u \,\mathrm{d}\bs x,
\qquad u \in W^{1,p}_0(\Omega). \label{eq:p-energy}
\end{equation}
Strict convexity of $J$ on $W^{1,p}_0(\Omega)$ ensures the existence and uniqueness of a weak solution $u^* \in W^{1,p}_0(\Omega)$.
Since nonzero Dirichlet boundary data $g$ can be treated by extending $g$ to the whole domain $\Omega$ and then solving for a corrective term to that extension, we assume $g=0$ for the remainder of this section.

Applying Newton's method directly to \eqref{eq:p-laplace} results in the linear equation \eqref{eq:newton-intro} that must be solved for the increment $\hat u$. Since this equation is only well defined if $\nabla u\not=0$ on all of $\Omega$, one often considers the regularized norm 
\begin{equation}\label{eq:eps-norm}
|\nabla u|_\varepsilon:=\sqrt{|\nabla u|^2+\varepsilon^2}
\end{equation}
in \eqref{eq:p-laplace} and \eqref{eq:p-energy}, with $0<\varepsilon\ll 1$. This results in a well-defined Newton step of the form \eqref{eq:newton-intro} with all norms $|\cdot|$ replaced by $|\cdot|_\varepsilon$. Regularizing in this (or a similar) way to avoid the singularity is essentially the only standard remedy for the ill-posedness of \eqref{eq:newton-intro}; if the resulting Newton iteration nonetheless converges poorly (as it typically does), it is unclear how to modify it further.

To isolate the nonlinearity, we thus introduce the $\mathbb R^d$-valued variable  $\bs\lambda(\bs x):= |\nabla u(\bs x)|^{p-2}\, \nabla u(\bs x) $ for all $\bs x\in\Omega$.
The resulting system in the independent variables $(u,\bs\lambda)$ is now
\begin{alignat}{2}
-\nabla\!\cdot\!\bs\lambda &= f \quad&&\text{in }\Omega,\label{eq:extrem1}\\
\bs\lambda - |\nabla u|^{p-2}\nabla u &= 0 \quad&&\text{in }\Omega,\label{eq:extrem2}
\end{alignat}
where we emphasize that \eqref{eq:extrem1} is linear, and the nonlinearity is restricted to the pointwise equation \eqref{eq:extrem2}, which can be manipulated to potentially change the nature of the nonlinearity. Before proceeding with a more detailed discussion of the solution to the lifted system \eqref{eq:extrem1}, \eqref{eq:extrem2}, we rigorously re-derive it below using Fenchel duality theory. We emphasize, however, that it is not necessary for the lifting variable $\bs\lambda$ to admit an interpretation as the solution of a dual problem.

\subsection{Dual problem and extremality conditions}\label{sec:duality}
To compute the dual optimization problem corresponding to \eqref{eq:p-energy} following \cite{EkelandTemam99}, we write $J(u) = F(\nabla u) + G(u)$ with $F(\bs\xi) = \tfrac{1}{p}\int_\Omega|\bs\xi|^p\,\mathrm{d}\bs x$ and $G(u) = -\int_\Omega f u\,\mathrm{d}\bs x$. Fenchel duality gives the dual functions
$F^\star(\bs\lambda) = \frac{1}{q}\int_\Omega |\bs\lambda|^q\,\mathrm{d}\bs x$ and $G^\star(\nabla \cdot \bs\lambda)=0$ if $-\nabla\cdot\bs\lambda - f=0$ and $G^\star(\nabla \cdot \bs\lambda)=\infty$ otherwise. Thus, the dual problem is
\begin{equation*}
\inf_{\bs \lambda}\;\frac{1}{q}\int_\Omega|\bs\lambda|^q\,\mathrm{d}\bs x \quad\text{subject to}\quad -\nabla\!\cdot\!\bs\lambda = f \;\text{ in }\Omega,
\end{equation*}
with $\bs\lambda$ in $W^{q}(\mathrm{div};\Omega)$.
The corresponding extremality conditions connecting the primal and dual solutions are now exactly \eqref{eq:extrem1}, \eqref{eq:extrem2}, derived in an ad-hoc fashion before.

This nonlinear system in $u$ and $\bs \lambda$ can be written in several equivalent forms. For instance, recognizing that \eqref{eq:extrem2} implies $|\bs \lambda| = |\nabla u|^{p-1}$, \eqref{eq:extrem2} is equivalent to
$\nabla u = |\bs\lambda|^{q-2}\bs\lambda$.
A similar, equivalent form for \eqref{eq:extrem2} is
\begin{equation}\label{eq:extrem2-mod}
|\nabla u|^{2-p}\bs\lambda - \nabla u = 0.    
\end{equation}
While these forms are mathematically equivalent, they include pointwise nonlinear transformations. As a consequence, their linearizations, and thus the corresponding Newton methods, differ. We will, in the following, use \eqref{eq:extrem2-mod}.

\subsection{Regularization and Newton methods}
\label{ssec:newton}
To avoid the singularity that occurs when $\nabla u = 0$ and $1 < p < 2$, we use the regularized norm $|\cdot|_\varepsilon$ from \eqref{eq:eps-norm} in the development of our Newton-type methods. Replacing $|\cdot|$ by $|\cdot|_\varepsilon$ in \eqref{eq:p-energy} results in the modified objective 
\begin{equation}\label{eq:J-eps}
  J_\varepsilon(u):=\frac1p\int_\Omega|\nabla u|_\varepsilon^{p}\,\mathrm{d}\bs x-\int_\Omega f\,u\,\mathrm{d}\bs x,
\end{equation}
and taking first variations yields the following functional $r\in W^{-1,q}(\Omega)$, i.e., the $p$-Laplace equation residual: 
\begin{equation}
\label{eq:p-laplace-gradient}
r:=-\nabla \cdot \bigl( |\nabla u|_\varepsilon^{p-2}\, \nabla u \bigr) - f .
\end{equation}
Integrating by parts shows that, for every $v\in W^{1,p}_0(\Omega)$,
\begin{equation}\label{eq:residual-is-derivative}
  \langle r,v\rangle
  =\int_\Omega |\nabla u|_\varepsilon^{p-2}\nabla u\cdot\nabla v\,\mathrm{d}\bs x
   -\int_\Omega f\,v\,\mathrm{d}\bs x
  = J_\varepsilon'(u)v,
\end{equation}
i.e., $r$ is the derivative of $J_\varepsilon$, regarded as an element of $W^{-1,q}(\Omega)$. In
the following we use whichever of the two viewpoints is more convenient: the energy
derivative $J_\varepsilon'$ where descent properties are at issue, and the residual $r$ when
reporting computed quantities.

Introducing an independent variable $\bs\lambda$ as above results in
the following regularized counterparts of \eqref{eq:extrem1}, \eqref{eq:extrem2}:
\begin{alignat}{2}
-\nabla\!\cdot\!\bs\lambda &= f \quad &&\text{in }\Omega,\label{eq:extrem1-reg}\\
\bs\lambda - |\nabla u|_\varepsilon^{p-2}\nabla u &= 0 &&\text{in }\Omega.\label{eq:extrem2-reg}
\end{alignat}
Here, the regularized flux \eqref{eq:extrem2-reg} is defined directly rather than through a dual problem; as noted above, such an interpretation is not required for the lifting.
Since $|\nabla u|_\varepsilon\geq\varepsilon>0$ for all $\bs x\in\Omega$, \eqref{eq:extrem2-reg} is equivalent to
\begin{equation}\label{eq:extrem2-mod-reg}
|\nabla u|_\varepsilon^{2-p}\bs\lambda - \nabla u=0.
\end{equation}
We introduce the corresponding normalized directions
\begin{equation}\label{eq:w-directions}
  \bs w_\varepsilon := \frac{\nabla u}{|\nabla u|_\varepsilon},
  \qquad\qquad
  \bs w_{\bs\lambda} := \frac{\bs\lambda}{|\nabla u|_\varepsilon^{\,p-1}},
\end{equation}
and refer to $\bs w_\varepsilon$ as the \emph{primal direction} and to $\bs w_{\bs\lambda}$ as the \emph{dual direction}. The former is the regularized counterpart of the unit gradient direction $\bs w_k$ used in \eqref{eq:newton-intro}; throughout this section, we suppress the iteration index. Note that $\bs w_{\bs\lambda}=\bs w_\varepsilon$ precisely when $u$ and $\bs\lambda$ satisfy the optimality condition \eqref{eq:extrem2-reg}.

Treating $u$ and $\bs \lambda$ as independent variables in \eqref{eq:extrem1-reg} and \eqref{eq:extrem2-mod-reg}, we can now derive the Newton step for this lifted system for the increments $(\hat u,\hat{\bs\lambda})$:
\begin{align}
  -\nabla\cdot \hat{\bs \lambda} &= f + \nabla\cdot\bs \lambda, \label{eq:newton1-lift}\\
  \bigl[(2-p)\,\bs w_{\bs\lambda}\bs w_\varepsilon^\top - I\bigr]\nabla \hat{u}
  + |\nabla u|_\varepsilon^{2-p}\hat{\bs \lambda}
  &= \nabla u - |\nabla u|_\varepsilon^{2-p}\bs \lambda.\label{eq:newton2-lift}
\end{align}
The right hand side of \eqref{eq:newton2-lift} admits a simple interpretation in terms of the normalized directions \eqref{eq:w-directions}, namely
\begin{equation*}
  \nabla u - |\nabla u|_\varepsilon^{2-p}\bs\lambda = |\nabla u|_\varepsilon\bigl(\bs w_\varepsilon - \bs w_{\bs\lambda}\bigr).
\end{equation*}
Up to the scaling by $|\nabla u|_\varepsilon$, the residual of the optimality condition is thus precisely the mismatch between the primal and the dual directions, and the lifted Newton system is driven by how far the dual direction lags behind the primal one.

The system \eqref{eq:newton1-lift}, \eqref{eq:newton2-lift} is formulated in terms of the lifted increments $(\hat u, \hat{\bs\lambda})$ and therefore has a higher dimension than the original $p$-Laplace equation. To avoid computing the Newton increment in the lifted space, we isolate the flux increment $\hat{\bs\lambda}$ in equation \eqref{eq:newton2-lift} and obtain the pointwise relation
\begin{equation}\label{eq:lambda-update}
  \hat{\bs \lambda} = |\nabla u|_\varepsilon^{p-2} \left(I + (p-2)\,\bs w_{\bs\lambda}\bs w_\varepsilon^\top\right)  \nabla \hat u - \bs \lambda + |\nabla u|_\varepsilon^{p-2} \nabla u.
\end{equation}
Using this in \eqref{eq:newton1-lift} and applying simple algebraic simplification results in a PDE in the Newton increment $\hat u$ only:
\begin{align}\label{eq:lifted-linear-sys}
        -\nabla \cdot \Bigl(|\nabla u|_\varepsilon^{p-2}\left(I + (p-2)\,\bs w_{\bs\lambda}\bs w_\varepsilon^\top\right)  \nabla \hat u \Bigr) = f + \nabla \cdot \left( |\nabla u|_\varepsilon^{p-2} \nabla u\right).
\end{align}
Note that the right hand side of this equation is the negative equation residual $r$ defined in \eqref{eq:p-laplace-gradient}.
Thus, if $u$ solves that equation, the increment $\hat u$ vanishes, as expected.

We compare this to the equation for the Newton increment for the original $p$-Laplace equation without lifting, using the same regularization:
    \begin{align}\label{eq:p-Laplace-Newton}
        -\nabla \cdot \Bigl(|\nabla u|_\varepsilon^{p-2}\left(I + (p-2)\,\bs w_\varepsilon\bs w_\varepsilon^\top\right)  \nabla \hat u \Bigr) = f + \nabla \cdot \left( |\nabla u|_\varepsilon^{p-2} \nabla u\right).
    \end{align}
Comparing \eqref{eq:lifted-linear-sys} with \eqref{eq:p-Laplace-Newton}, the two systems share the same right hand side. They only differ in one of the factors of the rank-one tensor: the lifted method uses the lagged dual direction $\bs w_{\bs\lambda}$, while standard Newton uses the primal direction $\bs w_\varepsilon$. In particular, the two linear systems coincide as soon as $\bs\lambda = |\nabla u|_\varepsilon^{p-2}\nabla u$, which holds at convergence. However, as our numerical results will show, this minor change in \eqref{eq:lifted-linear-sys} has a significant impact on the robustness of the convergence.

Note further that applying the above lift--linearize--eliminate procedure, without the transformation step, directly to \eqref{eq:extrem1-reg}, \eqref{eq:extrem2-reg} rather than to \eqref{eq:extrem1-reg}, \eqref{eq:extrem2-mod-reg} results in the Newton step \eqref{eq:p-Laplace-Newton}, i.e., the iteration one obtains by applying Newton's method directly to the (regularized) $p$-Laplace equation. The reason is that the lifted system is linear in $\bs\lambda$, and \eqref{eq:extrem2-reg} can be solved explicitly for $\bs\lambda$. Condensing the corresponding Newton system simply reproduces the chain rule, and one recovers the Newton step of the reduced equation. Introducing the lifting variable alone, therefore, has no effect on the iteration. 

What does change the iteration is using \eqref{eq:extrem2-mod-reg} instead of \eqref{eq:extrem2-reg}, which multiplies the optimality condition by the factor $|\nabla u|_\varepsilon^{2-p}$. Newton's method is invariant under affine transformations $F\mapsto AF$ of the residual with a \emph{constant} invertible $A$ \cite{Deuflhard04}. Here, however, the factor depends on the current iterate, so this {\em nonlinear} rescaling falls outside the class of affine transformations and produces the genuinely different iteration \eqref{eq:lifted-linear-sys}. It is thus not the lifting itself but the nonlinear transformation performed \emph{before} linearizing that yields the new method.

\subsection{Modifications to ensure well-posedness}\label{ssec:modifications}
Since $\bs w_\varepsilon$ defined in \eqref{eq:w-directions} satisfies $|\bs w_\varepsilon|
\le1$, it can be seen that for the standard Newton linearization \eqref{eq:p-Laplace-Newton}, the diffusion tensor $|\nabla u|_\varepsilon^{p-2}\left(I + (p-2)\,\bs w_\varepsilon\bs w_\varepsilon^\top\right)\in \mathbb R^{d\times d}$ is symmetric and positive definite for all $\bs x\in \Omega$.
As a consequence, \eqref{eq:p-Laplace-Newton}, combined with appropriate boundary conditions, is a well-defined elliptic PDE. This is not automatically the case for \eqref{eq:lifted-linear-sys}. In particular, the diffusion tensor
\begin{equation}\label{eq:lifted-tensor}
    \left(I + (p-2)\,\bs w_{\bs\lambda}\bs w_\varepsilon^\top\right)
\end{equation}
may fail to be positive definite, since $p-2<0$ and the norm of $\bs w_{\bs\lambda}$ is no longer bounded by one; in addition, it is not symmetric unless $\bs w_{\bs\lambda}$ and $\bs w_\varepsilon$ are parallel.
We address these two issues by suitably modifying the PDE operator. Note that altering the linear operator on the left-hand side of a Newton system may change the local rate of convergence. However, it does not change the underlying equation being solved, as long as its right hand side is given by the (negative) equation residual \eqref{eq:p-laplace-gradient},
as in \eqref{eq:lifted-linear-sys}.

The essential modification concerns positive definiteness. For $\bs z\in\mathbb R^d$, the Cauchy--Schwarz inequality gives
\begin{equation}\label{eq:tensor-quadratic-form}
  \bs z^\top\bigl(I + (p-2)\,\bs w_{\bs\lambda}\bs w_\varepsilon^\top\bigr)\bs z
  = |\bs z|^2 + (p-2)\,(\bs w_{\bs\lambda}\cdot\bs z)(\bs w_\varepsilon \cdot \bs z)
  \;\ge\; \bigl(1-(2-p)\,|\bs w_{\bs\lambda}|\,|\bs w_\varepsilon|\bigr)|\bs z|^2 .
\end{equation}
Since $|\bs w_\varepsilon|\le 1$, positive definiteness is assured when $(2-p)|\bs w_{\bs\lambda}|<1$. We therefore enforce, for all $\bs x\in\Omega$, the pointwise \emph{feasibility bound} $|\bs w_{\bs\lambda}(\bs x)|\le{\beta}/{(2-p)}$, or equivalently
\begin{equation}\label{eq:safeguard}
  |\bs\lambda(\bs x)|\le\frac{\beta}{2-p}\,|\nabla u(\bs x)|_\varepsilon^{p-1},
\end{equation}
for a parameter $\beta<1$, which we restrict further below. Combining \eqref{eq:safeguard} with \eqref{eq:tensor-quadratic-form} yields the two-sided bound
\begin{equation}\label{eq:tensor-sandwich}
  (1-\beta)\,|\bs z|^2
  \;\le\; \bs z^\top\bigl(I + (p-2)\,\bs w_{\bs\lambda}\bs w_\varepsilon^\top\bigr)\bs z
  \;\le\; (1+\beta)\,|\bs z|^2 ,
\end{equation}
so that the diffusion tensor of \eqref{eq:lifted-linear-sys} is uniformly positive definite whenever $\beta<1$.

The parameter $\beta$ should not be chosen arbitrarily small. At convergence, we have $\bs w_{\bs\lambda}=\bs w_\varepsilon$ and $|\bs w_\varepsilon|
\to 1$ as $\varepsilon\to0$, so a value of $\beta$ with $\beta/(2-p)<1$ would be active even at the solution of \eqref{eq:extrem1-reg}, \eqref{eq:extrem2-reg}. This would %
destroy the consistency $\bs w_{\bs\lambda}=\bs w_\varepsilon$ that is desirable for local fast convergence. We therefore require $\beta/(2-p)>1$, which, together with $\beta<1$, leaves the admissible, always nonempty window
\begin{equation}\label{eq:beta-window}
  2-p<\beta<1.
\end{equation}
In practice, the feasibility bound is imposed by projecting the lifting variable;
see \eqref{eq:discrete-lambda-projection}.

The second modification is not needed for well-posedness, but is convenient in practice. It consists of
replacing the rank-one matrix in \eqref{eq:lifted-tensor} by its symmetric part, i.e., using
\begin{equation}\label{eq:lifted-tensor-sym}
  I + (p-2)\,\sym\bigl(\bs w_{\bs\lambda}\bs w_\varepsilon^\top\bigr),
  \qquad \sym(B):=\frac12\bigl(B+B^\top\bigr).
\end{equation}
This leaves the associated quadratic form---and hence the bounds \eqref{eq:tensor-quadratic-form} and \eqref{eq:tensor-sandwich} established above---unchanged. It does, however, render the discretized operator symmetric, allowing the use of symmetric solvers and preconditioners for the resulting linear systems. Observe that this modification has no effect at a consistent pair, where $\bs w_{\bs\lambda}=\bs w_\varepsilon$ and \eqref{eq:lifted-tensor} is already symmetric.

Note that \eqref{eq:tensor-sandwich} is a statement about well-posedness, not about conditioning: since the lifted and the standard linearizations coincide at the solution, the two methods cannot differ in their asymptotic conditioning. What \eqref{eq:tensor-sandwich} does provide is a linear system that remains solvable, and an increment that remains a descent direction for $J_\varepsilon$, no matter how far the dual direction lags behind the primal one. This is what underlies the global convergence result of the next section.

\subsection{Summary of lifted Newton algorithm}
\label{ssec:lifted-algorithm}
Collecting the ingredients from the previous subsections yields the method we analyze and test below. We write
$\bs w_{\varepsilon,k}:=\nabla u_k/|\nabla u_k|_\varepsilon$ and
$\bs w_{\bs\lambda,k}:=\bs\lambda_k/|\nabla u_k|_\varepsilon^{\,p-1}$
for the directions \eqref{eq:w-directions} evaluated at the $k$-th iterate, and recall the regularized objective \eqref{eq:J-eps} with corresponding (lifted) stationarity conditions \eqref{eq:extrem1-reg}, \eqref{eq:extrem2-reg}.

We assume that $\beta$ satisfies \eqref{eq:beta-window}, and enforce the feasibility bound only at $\bs x\in\Omega$ where it is needed, namely in the construction of the diffusion tensor. We define
\begin{equation}\label{eq:discrete-lambda-projection}
  \bs\lambda_k^{\mathrm{pr}}(\bs x):=\frac{\bs\lambda_k(\bs x)}{\max\bigl(1,\,|\bs\lambda_k(\bs x)|/\rho_k(\bs x)\bigr)},
  \qquad \rho_k(\bs x):=\frac{\beta}{2-p}\,|\nabla u_k(\bs x)|_\varepsilon^{p-1},
\end{equation}
so that $|\bs\lambda_k^{\mathrm{pr}}|=\min\bigl(|\bs\lambda_k|,\rho_k\bigr)$ pointwise.
The associated direction $\bs w_{\bs\lambda,k}^{\mathrm{pr}}:=\bs\lambda_k^{\mathrm{pr}}/|\nabla u_k|_\varepsilon^{\,p-1}$ is thus the projection of $\bs w_{\bs\lambda,k}$ onto the ball of radius $\beta/(2-p)$, so that \eqref{eq:safeguard} holds at every iterate and for every $\bs\lambda_k$.
The method then proceeds as follows.
\begin{enumerate}[leftmargin=*,label=(\roman*)]
\item \emph{Initialize}: set $k=0$, initialize $u_0$ such that it satisfies the desired nonhomogeneous Dirichlet boundary conditions as in \eqref{eq:p-laplace}, initialize $\bs\lambda_0$, and choose the Armijo parameter $\gamma\in(0,1)$.
\item \emph{Assemble and solve}: use the symmetrized diffusion tensor \eqref{eq:lifted-tensor-sym} with the projected direction,
\begin{equation}\label{eq:Mk}
  M_k:=|\nabla u_k|_\varepsilon^{p-2}\Bigl(I+(p-2)\,\sym\bigl(\bs w_{\bs\lambda,k}^{\mathrm{pr}}\bs w_{\varepsilon,k}^\top\bigr)\Bigr),
\end{equation}
to assemble and solve the system for the primal increment $\hat u_k$, i.e., solve \eqref{eq:lifted-linear-sys} with diffusion tensor $M_k$ and homogeneous Dirichlet boundary conditions.
\item \emph{Recover}: compute the lifting variable increment pointwise from \eqref{eq:lambda-update},
\begin{equation}\label{eq:lambda-hat}
  \hat{\bs\lambda}_k:=M_k\nabla\hat u_k-\bs\lambda_k+|\nabla u_k|_\varepsilon^{p-2}\nabla u_k.
\end{equation}
\item \emph{Line search}: choose
$\alpha_k$ as the largest number in $\{1,1/2,1/2^2,\ldots\}$ with
\begin{equation}\label{eq:armijo}
  J_\varepsilon(u_k+\alpha_k\hat u_k)\le J_\varepsilon(u_k)+\gamma\,\alpha_k\,J_\varepsilon'(u_k)\hat u_k.
\end{equation}
\item \emph{Update and advance}: set $u_{k+1}:=u_k+\alpha_k\hat u_k$ and $\bs\lambda_{k+1}:=\bs\lambda_k+\alpha_k\hat{\bs\lambda}_k$. Set $k\leftarrow k+1$ and return to step~(ii), until a stopping criterion is met.
\end{enumerate}
Several things should be pointed out. First, eliminating $\hat{\bs\lambda}$ from the lifted system cancels $\bs\lambda_k$ from the right-hand side of the system solved in step~(ii), so the lifting variable influences the primal iterate \emph{only} through $M_k$. Projecting it there is therefore enough, and no constraint needs to be placed on $\bs\lambda_k$ itself. One may also project $\bs\lambda_k$ after each update; this variant of the algorithm is briefly discussed in Appendix \ref{app:project}.
Second, $M_k$ is symmetric and, by \eqref{eq:tensor-sandwich}, uniformly positive definite for every $k$ and every initial guess $\bs\lambda_0$. Since $\beta>2-p$, the projection \eqref{eq:discrete-lambda-projection} is inactive whenever $\bs\lambda_k$ is close to satisfying \eqref{eq:extrem2-reg}, and in particular in a neighborhood of the solution; it acts only during the initial phase of the iteration.
Third, initializing the lifting variable by $\bs\lambda_0=\bs 0$ makes the rank-one term vanish, so that $M_0=|\nabla u_0|_\varepsilon^{p-2}I$ and the first iterate is a Picard step \eqref{eq:picard-intro}, known to be a robust choice.
Note finally that only step~(ii) requires a linear solve; steps (iii) and (v) are pointwise algebraic operations on $\nabla u_k$ and $\bs\lambda_k$, so the cost of one lifted Newton iteration is essentially that of one standard Newton (or Picard) iteration.

\section{Analysis of the lifted method}\label{sec:analysis}
In this section, we analyze the lifted Newton iteration of \cref{ssec:lifted-algorithm} for the $p$-Laplace equation with $1<p<2$.
We work in the discretized setting to show global convergence with line search and fast local convergence in \cref{ssec:p-Lapl-convergence,ssec:local-convergence}. Both results also hold for the standard Newton method. The greater robustness of lifted Newton away from the solution and for $p\to 1$ is not captured by these results. In \cref{ssec:robustness}, we discuss the mechanism behind this difference, which we make precise for a one-dimensional model problem in Appendix~\ref{app:one_dim_flux}.

\subsection{Global convergence for discretized method}
\label{ssec:p-Lapl-convergence}
Proving global convergence of Newton's methods to solve nonlinear PDEs such as the $p$-Laplace equation in infinite dimensions is very challenging due to the lack of Hilbert space structure. One remedy, followed in \cite{schmidt-pStokes-global} for the $p$-Stokes system, is to use Hilbert space regularization. That is, 
in addition to the regularization $|\nabla u|_\varepsilon$, a term $\mu_0\|\nabla u\|_{L^2}^2$ with $\mu_0>0$ is added to the energy, which places the problem in $H^1$.
We take a different approach and establish global convergence for the discretized problem, that is, in a finite-dimensional setting: after fixing $V_h$, compactness of the discrete energy sublevel set provides the coercivity constant that the $H^1$ term provides in \cite{schmidt-pStokes-global}. The proof follows the same Armijo framework. Although this avoids the need for extra regularization, it comes with the drawback that the constants may depend on the discretization.

Let $V_h\subset W^{1,\infty}(\Omega)$ be a fixed conforming finite element space and let
$ \mathring{V}_h:=V_h\cap W_0^{1,2}(\Omega)$.
We equip $\mathring{V}_h$ with the $H^1$ seminorm
$\|u_h\|_h:=\|\nabla u_h\|_{L^2(\Omega)}$.
We use $\mathring{V}_h^\star$ to denote the dual space of $\mathring{V}_h$, and denote its norm by $\|\cdot\|_{h,\star}$.
We study the lifted Newton iteration of \cref{ssec:lifted-algorithm}, discretized in $\mathring{V}_h$. We consider an extension of the Dirichlet data $g$ to the domain $\Omega$ and an iteration $u_k\in g+\mathring{V}_h$ with a lifting variable $\bs\lambda_k$. Then, the increment $\hat u_k\in \mathring{V}_h$ of step~(ii) is determined by solving the finite element approximation of \eqref{eq:lifted-linear-sys} with the diffusion tensor
$M_k$ as in \eqref{eq:Mk}; the remaining steps (iii)--(v) are carried out unchanged. We denote by $r_{h,k}\in\mathring{V}_h^\star$ the discrete residual at the iterate $u_{h,k}$, i.e., the restriction to $\mathring{V}_h$ of the functional \eqref{eq:p-laplace-gradient}. By \eqref{eq:residual-is-derivative} it satisfies $\langle r_{h,k},v_h\rangle=J_\varepsilon'(u_{h,k})v_h$ for all $v_h\in\mathring{V}_h$. Since $\|v_h\|_h=\|\nabla v_h\|_{L^2(\Omega)}$, the dual norm $\|\cdot\|_{h,\star}$ used below is the $H^{-1}$ norm reported in the numerical experiments of \cref{sec:numerics-p-laplace}.
A global convergence result in finite dimensions, similar as presented below, also applies to the standard Newton method. Thus, the main point of the following theorem is that a global convergence analysis also holds for the lifted Newton method from \cref{ssec:lifted-algorithm}.

\begin{theorem}[Global convergence of the lifted method]
\label{thm:finite-dimensional-lifted-global-convergence}
Let $\beta$ satisfy \eqref{eq:beta-window} and assume arbitrary initializations $u_{h,0}\in g+\mathring{V}_{h}$
and 
$\bs\lambda_{h,0}$. 
Then the (discrete) iterates produced by the iteration of \cref{ssec:lifted-algorithm} converge globally, in the sense that the discrete residual
$r_{h,k}\to 0$ in $\mathring{V}_{h}^\star$,
and
\begin{equation*}
  u_{h,k}\to u_h^* \quad\text{in }V_h,
\end{equation*}
where $u_h^*\in g+\mathring{V}_{h}$ is the unique minimizer of $J_\varepsilon$ over $g+\mathring{V}_{h}$.
\end{theorem}
\begin{proof}
Consider the sublevel set
$\mathcal L_h:=\{u_h\in g+\mathring{V}_h:J_\varepsilon(u_h)\le J_\varepsilon(u_0)\}$, which is bounded because
$J_\varepsilon(u_h)\to \infty$ as $\|u_h\|_h\to \infty$. %
Since $V_h$ is finite-dimensional, $\mathcal{L}_h$ is compact, and therefore
$ C_h:=\max_{u_h\in\mathcal L_h}\|\nabla u_h\|_{L^\infty(\Omega)}<\infty$.
We prove the uniform boundedness and coercivity of the linearization resulting from the lifting approach on the sublevel set. To this end, suppose that $u_{h,k}\in\mathcal L_h$.
We first derive pointwise spectral bounds for $M_k$ defined in \eqref{eq:Mk}. For any $\bs z\in\mathbb R^d$,
\eqref{eq:tensor-sandwich} gives
\begin{equation}\label{eq:spectral-lower-upper}
  \begin{aligned}
    \bs z^\top M_k \bs z
    &\geq (1-\beta)|\nabla u_{h,k}|_\varepsilon^{p-2}|\bs z|^2
    \geq \underline m_h
    |\bs z|^2,\\
  \bs z^\top M_k \bs z&\leq (1+\beta)|\nabla u_{h,k}|_\varepsilon^{p-2}|\bs z|^2\leq
  \overline m_h |\bs z|^2,
  \end{aligned}
\end{equation}
where $\underline m_h:=(1-\beta)(C_h^2+\varepsilon^2)^{(p-2)/2}$ and  $\overline m_h:=(1+\beta)\varepsilon^{p-2}$.
Denoting
\begin{equation*}
  a_k(v_h,w_h):=\int_\Omega M_k\nabla v_h\cdot\nabla w_h\,\mathrm{d}\bs x,
\end{equation*}
we thus have
\begin{equation}\label{eq:Mk-coercive-bounded}
  a_k(v_h,v_h)\geq \underline m_h\|v_h\|_h^2, \quad |a_k(v_h,w_h)| \leq \overline m_h\|v_h\|_h\|w_h\|_h \quad\text{ for all } v_h,w_h\in\mathring{V}_h.
\end{equation}
The pointwise bounds \eqref{eq:Mk-coercive-bounded} imply that $a_k$ defines an inner product on $\mathring V_h$, so the primal increment is uniquely defined by $a_k(\hat u_{h,k},v_h)=-J_\varepsilon'(u_{h,k})v_h$ for all $v_h\in\mathring V_h$. Choosing $v_h=\hat u_{h,k}$ gives
\begin{equation}\label{eq:direct-descent-identity}
  J_\varepsilon'(u_{h,k})\hat u_{h,k}
  =-a_k(\hat u_{h,k},\hat u_{h,k}).
\end{equation}
If $\hat u_{h,k}=0$, then $J_\varepsilon'(u_{h,k})=0$ on
$\mathring V_h$, and strict convexity shows that $u_{h,k}=u_h^*$.  We may
therefore assume below that $\hat u_{h,k}\neq0$.

Next, we bound the energy along the search direction. Since $p<2$, the map $t\mapsto t^{p/2}$ is concave and thus lies below its tangent line at any $t_0>0$.
Applying this pointwise with $t_0=|\nabla u_h|_\varepsilon^2$ and $t=|\nabla u_h+\alpha\nabla v_h|_\varepsilon^2$, whose difference is $2\alpha\nabla u_h\!\cdot\!\nabla v_h+\alpha^2|\nabla v_h|^2$, then dividing by $p$, integrating, and adding the linear load term to both sides gives
\begin{equation}\label{eq:energy-majorization-discrete}
  J_\varepsilon(u_h+\alpha v_h)
  \le J_\varepsilon(u_h)+\alpha J_\varepsilon'(u_h)v_h
  +\frac{\alpha^2}{2}\int_\Omega
  |\nabla u_h|_\varepsilon^{p-2}|\nabla v_h|^2\,\mathrm{d}\bs x
\end{equation}
for all $u_h\in g+\mathring V_h$, $v_h\in\mathring V_h$, and $\alpha\ge0$. Note that the last term in \eqref{eq:energy-majorization-discrete} is the bilinear form of the Picard iteration \eqref{eq:picard-intro}, which majorizes the curvature of $J_\varepsilon$ because $p<2$. The estimate \eqref{eq:energy-majorization-discrete} holds globally in $\alpha$ and involves no remainder term.

We now take $u_h=u_{h,k}$ and $v_h=\hat u_{h,k}$. The lower bound in \eqref{eq:tensor-sandwich} gives
\begin{equation*}
  \int_\Omega|\nabla u_{h,k}|_\varepsilon^{p-2}|\nabla\hat u_{h,k}|^2\,\mathrm{d}\bs x
  \le\frac{1}{1-\beta}\,a_k(\hat u_{h,k},\hat u_{h,k}),
\end{equation*}
so that \eqref{eq:energy-majorization-discrete} and \eqref{eq:direct-descent-identity} yield, for every $\alpha\ge0$,
\begin{equation}\label{eq:armijo-majorant}
  J_\varepsilon(u_{h,k}+\alpha\hat u_{h,k})
  \le J_\varepsilon(u_{h,k})
  -\alpha\Bigl(1-\frac{\alpha}{2(1-\beta)}\Bigr)
   a_k(\hat u_{h,k},\hat u_{h,k}).
\end{equation}
By \eqref{eq:direct-descent-identity}, the Armijo condition \eqref{eq:armijo} requires a decrease of at least $\gamma\alpha\,a_k(\hat u_{h,k},\hat u_{h,k})$. Comparing with \eqref{eq:armijo-majorant}, it is therefore satisfied as soon as the bracket is at least $\gamma$, that is,
whenever $0<\alpha\le2(1-\gamma)(1-\beta)$. Since the line search halves the trial step
starting from $\alpha=1$, it terminates with an accepted step size
$\alpha_k\ge\underline\alpha:=(1-\gamma)(1-\beta)>0$.
Since \eqref{eq:armijo} and \eqref{eq:direct-descent-identity} give $J_\varepsilon(u_{h,k+1})\le J_\varepsilon(u_{h,k})$, starting from
$u_{h,0}\in\mathcal L_h$, induction therefore proves that the iteration is
well defined and that every $u_{h,k}$ remains in $\mathcal L_h$.
Using \eqref{eq:direct-descent-identity}, \eqref{eq:armijo}, and the lower bound on $\alpha$, we obtain
\begin{equation*}
  J_\varepsilon(u_{h,k})-J_\varepsilon(u_{h,k+1})
  \ge \gamma\underline\alpha\,
  a_k(\hat u_{h,k},\hat u_{h,k}).
\end{equation*}
Since $J_\varepsilon$ is bounded below, summation over $k$ shows that $a_k(\hat u_{h,k},\hat u_{h,k})\to 0$. Cauchy--Schwarz in the  $a_k$-inner product and the upper bound in
\eqref{eq:Mk-coercive-bounded} give, for every $v_h\in\mathring V_h$,
\begin{equation*}
  |J_\varepsilon'(u_{h,k})v_h|^2
  =|a_k(\hat u_{h,k},v_h)|^2
  \le \overline m_h\, a_k(\hat u_{h,k},\hat u_{h,k})\|v_h\|_h^2,
\end{equation*}
so that $r_{h,k}\to0$ in $\mathring V_h^\star$ by \eqref{eq:residual-is-derivative}. Compactness of $\mathcal L_h$ and strict convexity of $J_\varepsilon$ then prove $u_{h,k}\to u_h^*$ in $V_h$.
\end{proof}

Next, we prove the convergence of the discretized lifting variable $\bs \lambda_{h,k}$. The discretized lifting variable is defined at the quadrature points by $\bs x_q$, $q=1,\ldots,N_Q$ and it is thus an element of $(\mathbb R^d)^{N_Q}$. We denote 
$\bs\lambda_{h,k;q}:=\bs\lambda_{h,k}(\bs x_q)$ and more details on the practical implementation are provided in \cref{ssec:plaplace-setup}. Convergence in a finite-dimensional space is equivalent to componentwise convergence, which is what we show below. For $\bs\xi\in\mathbb R^d$, we define
$  S(\bs\xi):=|\bs\xi|_\varepsilon^{p-2}\bs\xi$.

\begin{corollary}[Pointwise convergence of the lifting variable]
\label{cor:quadrature-flux-convergence}
Under the assumptions of \cref{thm:finite-dimensional-lifted-global-convergence}, for every $q=1,\ldots,N_Q$,
\begin{equation*}
  \bs\lambda_{h,k;q}
  \longrightarrow
  S\left(\nabla u_h^*(\bs x_q)\right)
  \quad\text{ in }\mathbb R^d.
\end{equation*}
\end{corollary}
\begin{proof}
For a fixed $q$, let $S_k:=S(\nabla u_{h,k}(\bs x_q))$, and $S^*:=S\left(\nabla u_h^*(\bs x_q)\right)$. Since cell-local gradient
evaluation is continuous on the fixed finite-dimensional space $V_h$, and
$S$ is continuous, the convergence $u_{h,k}\to u_h^*$ gives $\nabla u_{h,k}(\bs x_q)\to \nabla u_h^*(\bs x_q)$, and thus
\begin{equation*}
  S_k\to S^*,
  \qquad
  \nabla(u_{h,k+1}-u_{h,k})(\bs x_q)\to0.
\end{equation*}
At the same quadrature point, \eqref{eq:lambda-hat} and step~(v) of
\cref{ssec:lifted-algorithm} give
\begin{align*}
  \bs\lambda_{h,k+1;q}-S_{k+1}
  =(1-\alpha_k)(\bs\lambda_{h,k;q}-S_k)+(S_k-S_{k+1})+M_k(\bs x_q)\nabla(u_{h,k+1}-u_{h,k})(\bs x_q).
\end{align*}
By the upper bound in \eqref{eq:spectral-lower-upper}, for every $\bs z\in\mathbb R^d$, we have $|M_k(\bs x_q)\bs z|\le \overline m_h|\bs z|$. Moreover, the proof of \cref{thm:finite-dimensional-lifted-global-convergence} showed that
$\alpha_k\ge\underline\alpha>0$.
Thus, with
\begin{equation*}
    e_k:=|\bs\lambda_{h,k;q}-S_k|, \qquad b_k:=|S_{k+1}-S_k|+\overline m_h\left|\nabla(u_{h,k+1}-u_{h,k})(\bs x_q)\right|
\end{equation*}
we have
\begin{equation}\label{eq:Ek-recursion}
    e_{k+1} \le (1-\underline\alpha)e_k+b_k.
\end{equation}
Here $b_k\to0$ by the convergence statements above. Since in addition $1-\underline\alpha<1$, this gives $e_k\to0$, and the pointwise convergence now follows from $S_k\to S^*$.
\end{proof}
Note that the coercivity constant $\underline m_h$ depends on the bound $C_h$, which comes from the compactness of the fixed finite element sublevel set. In the continuous space, an energy sublevel set gives an $L^p$-bound but not an $L^\infty$-bound on $\nabla u$. Therefore, the same argument does not yield a mesh-independent coercivity constant for $M_k$ on $H_0^1(\Omega)$ unless one adds an additional uniformly elliptic term as in \cite{schmidt-pStokes-global} or assumes a priori an $L^\infty$ bound on the gradients.

\subsection{Local convergence for discretized problem}
\label{ssec:local-convergence}
Recall that the lifted method was derived as a Newton method in lifted space, and that  
the resulting Newton method coincides with the standard Newton method at the solution, where $\bs w_{\bs\lambda}=\bs w_\varepsilon$. For these reasons, one can expect the lifted method to inherit the fast local convergence of the standard Newton method. However,
it must be shown 
that the linear operator arising in the lifted Newton method approaches the exact one fast enough. Measuring the
primal error together with the inconsistency of the pair
$(u_{h,k},\bs\lambda_{h,k})$ in \eqref{eq:extrem2-reg},
\begin{equation}\label{eq:local-errors}
  \mathcal E_k:=\|u_{h,k}-u_h^*\|_h
  +\bigl\|\bs\lambda_{h,k}-S(\nabla u_{h,k})\bigr\|_{\infty},
\end{equation}
we show below that the rate is quadratic. We will refer to standard arguments from the literature and use two ingredients specific to the lifted method. The standard ingredients are the local convergence of Newton's method with a perturbed Jacobian \cite{Kelley95} and the acceptance of unit steps by the Armijo line search \cite{NocedalWright06}. The specific ingredients are a bound on the perturbation of the Jacobian in terms of the inconsistency $\delta_k$ of the pair $(u_{h,k},\bs\lambda_{h,k})$, and the observation that the lifting variable update reduces this inconsistency quadratically.

Below, all pointwise quantities ($\bs\lambda_{h,k}$, $M_k$, $\nabla u_{h,k}$) are evaluated at the quadrature points $\bs x_q$, and $\|\cdot\|_\infty$ is the maximum over these points; we write $\delta_k$ for the second term in \eqref{eq:local-errors}. Since $V_h$ is finite-dimensional, there is a constant $C_{\rm inv}$ with $\|\nabla v_h\|_\infty\le C_{\rm inv}\|v_h\|_h$ for all $v_h\in V_h$. Since $\varepsilon>0$, the map $S$ of corollary \ref{cor:quadrature-flux-convergence} is smooth, with
\begin{equation}\label{eq:S-prime}
  S'(\bs\xi)=|\bs\xi|_\varepsilon^{p-2}\bigl(I+(p-2)\,\bs w\bs w^\top\bigr),\qquad \bs w:=\bs\xi/|\bs\xi|_\varepsilon,
\end{equation}
which is the diffusion tensor of standard Newton \eqref{eq:p-Laplace-Newton}, and $\|S'(\bs\xi)\|\le\varepsilon^{p-2}$ and $\|S''(\bs\xi)\|\le C_p\,\varepsilon^{p-3}$ for all $\bs\xi\in\mathbb R^d$. We write $H_k:=S'(\nabla u_{h,k})$ for the standard Newton tensor at the $k$-th iterate. All constants below depend only on $h$, $\varepsilon$, $p$, $\beta$ and $\gamma$.

\begin{lemma}[Jacobian error]\label{lem:tangent-error}
For every $k$,
\begin{equation}\label{eq:tangent-error}
  \|M_k-H_k\|_\infty\le\frac{2-p}{\varepsilon}\,\delta_k .
\end{equation}
\end{lemma}
\begin{proof}
Comparing \eqref{eq:Mk} and \eqref{eq:S-prime},
$M_k-H_k=(p-2)|\nabla u_{h,k}|_\varepsilon^{p-2}\sym\bigl((\bs w_{\bs\lambda,k}^{\mathrm{pr}}-\bs w_{\varepsilon,k})\bs w_{\varepsilon,k}^\top\bigr)$.
Since $|\bs w_{\varepsilon,k}|\le1<\beta/(2-p)$ by \eqref{eq:beta-window}, the projection \eqref{eq:discrete-lambda-projection} leaves $\bs w_{\varepsilon,k}$ fixed and is $1$-Lipschitz, so that
$|\bs w_{\bs\lambda,k}^{\mathrm{pr}}-\bs w_{\varepsilon,k}|\le|\bs w_{\bs\lambda,k}-\bs w_{\varepsilon,k}|=|\bs\lambda_{h,k}-S(\nabla u_{h,k})|/|\nabla u_{h,k}|_\varepsilon^{p-1}$. Using $\|\sym(\bs a\bs b^\top)\|\le|\bs a||\bs b|$ and $|\nabla u_{h,k}|_\varepsilon\ge\varepsilon$ gives \eqref{eq:tangent-error}.
\end{proof}

The second ingredient concerns the lifting variable. Suppose that a unit step is taken, $\alpha_k=1$. Then \eqref{eq:lambda-hat} and step~(v) of \cref{ssec:lifted-algorithm} give
$\bs\lambda_{h,k+1}=\bs\lambda_{h,k}+\hat{\bs\lambda}_{h,k}=S(\nabla u_{h,k})+M_k\nabla\hat u_{h,k}$, whereas Taylor expansion of $S$ gives
$S(\nabla u_{h,k+1})=S(\nabla u_{h,k})+H_k\nabla\hat u_{h,k}+\bs R_k$ with $\|\bs R_k\|_\infty\le C_p\,\varepsilon^{p-3}\,\|\nabla\hat u_{h,k}\|_\infty^2$. Subtracting the two identities, the first-order terms cancel, and we obtain
\begin{equation}\label{eq:consistency-identity}
  \bs\lambda_{h,k+1}-S(\nabla u_{h,k+1})
  =(M_k-H_k)\nabla\hat u_{h,k}+\bs R_k .
\end{equation}
This cancellation makes the method second order: the lifting variable update \eqref{eq:lambda-hat} adds to $\bs\lambda_{h,k}$ exactly the first-order change of $S(\nabla u_{h,k})$, up to the Jacobian error. Together with \eqref{eq:tangent-error} and the inverse estimate, \eqref{eq:consistency-identity} yields
\begin{equation}\label{eq:delta-recursion}
  \delta_{k+1}\le C\bigl(\delta_k+\|\hat u_{h,k}\|_h\bigr)\|\hat u_{h,k}\|_h .
\end{equation}

\begin{theorem}[Local quadratic convergence]\label{thm:local-quadratic}
Let the assumptions of
\cref{thm:finite-dimensional-lifted-global-convergence} hold and let the
Armijo parameter satisfy $\gamma<\tfrac12$. Then there exist $\rho>0$ and
$C_\star>0$, independent of $k$, such that $\mathcal E_k\le\rho$ implies
$\alpha_k=1$ and
\begin{equation}\label{eq:quadratic}
  \mathcal E_{k+1}\le C_\star\,\mathcal E_k^2 .
\end{equation}
In particular, since $\mathcal E_k\to0$ by \cref{thm:finite-dimensional-lifted-global-convergence} and corollary \ref{cor:quadrature-flux-convergence}, the iteration eventually takes unit steps and converges quadratically.
\end{theorem}
\begin{proof}
Consider the discrete residual $F_h(u_h):=J_\varepsilon'(u_h)|_{\mathring V_h}$ on the finite-dimensional space $g+\mathring V_h$ with norm $\|\cdot\|_h$. Its Jacobian at $u_{h,k}$ is the operator induced by the bilinear form $\int_\Omega H_k\nabla v_h\cdot\nabla w_h\,\mathrm{d}\bs x$; it is Lipschitz continuous by the bound on $S''$ and the inverse estimate, and $F_h'(u_h^*)$ is invertible since $H_k$ is uniformly positive definite. The primal increment satisfies $\hat u_{h,k}=-B_k^{-1}F_h(u_{h,k})$, where $B_k$ is the operator induced by $a_k$, so the primal iteration is a Newton method with the perturbed Jacobian $B_k=F_h'(u_{h,k})+\Delta_k$, and $\|\Delta_k\|\le\|M_k-H_k\|_\infty\le\frac{2-p}{\varepsilon}\delta_k$ by lemma \ref{lem:tangent-error}. Moreover, by \eqref{eq:direct-descent-identity}, \eqref{eq:Mk-coercive-bounded}, $F_h(u_h^*)=0$ and $\|S'\|\le\varepsilon^{p-2}$,
\begin{equation}\label{eq:increment-bound}
  \|\hat u_{h,k}\|_h\le\frac{\varepsilon^{p-2}}{\underline m_h}\,\|u_{h,k}-u_h^*\|_h .
\end{equation}
Since $\gamma<\tfrac12$ and $\hat u_{h,k}$ is a descent direction by \eqref{eq:direct-descent-identity}, the sufficient decrease condition \eqref{eq:armijo} holds at $\alpha=1$ as soon as $\|u_{h,k}-u_h^*\|_h$ and the relative Jacobian perturbation $\|\Delta_k\hat u_{h,k}\|_{h,\star}/\|\hat u_{h,k}\|_h\le\|\Delta_k\|$ are small enough; see the proof of \cite[Theorem~3.6]{NocedalWright06}, which only uses the Armijo condition. Both quantities are bounded by $C\mathcal E_k$, so there is $\rho>0$ such that $\mathcal E_k\le\rho$ implies $\alpha_k=1$.
With $\alpha_k=1$, the perturbed-Jacobian estimate of \cite[Theorem~5.4.1]{Kelley95} gives, for $\mathcal E_k\le\rho$ (after decreasing $\rho$ if necessary),
\begin{equation*}
  \|u_{h,k+1}-u_h^*\|_h\le K\bigl(\|u_{h,k}-u_h^*\|_h^2+\|\Delta_k\|\,\|u_{h,k}-u_h^*\|_h\bigr)\le C\,\mathcal E_k^2 .
\end{equation*}
By \eqref{eq:delta-recursion} and \eqref{eq:increment-bound},
$\delta_{k+1}\le C(\delta_k+\|u_{h,k}-u_h^*\|_h)\|u_{h,k}-u_h^*\|_h\le C\,\mathcal E_k^2$.
Adding the two estimates gives \eqref{eq:quadratic}.
\end{proof}
The constant $C_\star$ degrades in the two limits that make the problem hard: through the factor $(2-p)/\varepsilon$ in \eqref{eq:tangent-error} and the factor $\varepsilon^{p-3}$ in \eqref{eq:consistency-identity} as $\varepsilon\to0$, and through $\underline m_h$, which carries $1-\beta<p-1$ by \eqref{eq:beta-window}, as $p\to1$. The same dependence on $\varepsilon$ and $p$ is present in the corresponding constants for the standard Newton method.

\subsection{Robustness away from the solution}
\label{ssec:robustness}
Since the results of \cref{ssec:p-Lapl-convergence,ssec:local-convergence} apply equally to standard Newton, the difference between the two methods observed in practice must arise away from the solution.
We next discuss the properties of the two linearizations, focusing on values of $p$ near $1$, where solving \eqref{eq:p-laplace} is particularly difficult.

\noindent
\emph{Prevalence of $\varepsilon$-flat regions.}
First, note that solutions of the $p$-Laplace equation exhibit largely flat regions, especially when $p$ is close to 1. This is evident from the energy formulation \eqref{eq:p-energy}, which converges to the $L^1$-norm of the gradient as $p\to 1$. Such objectives promote sparsity, meaning the gradient vanishes over large parts of the domain and variation concentrates on small sets. In the $\varepsilon$-regularized problem, the relevant notion of flatness corresponds to values of $|\nabla u|$ of an order of magnitude of $\varepsilon$ or smaller. %
This is the threshold used in the discussion below.

\noindent
\emph{Instability of standard Newton.}
Next, it is useful to separate the two distinct ways in which flat regions affect the linearizations. The first is through the scalar prefactor $|\nabla u|_\varepsilon^{p-2}$, 
which is common to all methods considered here and thus cannot account for differences between them.
The second is through the direction $\bs w_\varepsilon=\nabla u/|\nabla u|_\varepsilon$ entering the rank-one term. Where $|\nabla u|\gg\varepsilon$, this direction is influential but well determined and varies smoothly. In deeply flat regions, $|\nabla u|\ll\varepsilon$, it is poorly determined and highly sensitive to changes in $u$, but it is also of little influence since $|\bs w_\varepsilon|\le|\nabla u|/\varepsilon$ suppresses the rank-one term and the Newton methods reduce to the Picard operator. It is when $|\nabla u|\sim\varepsilon$ that $\bs w_\varepsilon$ is simultaneously influential and highly sensitive to small changes in $u$.
The two methods respond differently to this sensitivity. In the standard Newton method,
the rank-one term is $\bs w_\varepsilon\bs w_\varepsilon^\top $ and both of its factors are recomputed from the current iterate.
The Newton direction $\hat u$ may then vary strongly from one iteration to the next, a phenomenon known as zig-zag behavior, which typically makes substantial damping of the step size necessary for convergence. 
In the lifted method, $\bs\lambda$ is updated only through \eqref{eq:lambda-hat}, so the factor $\bs w_{\bs\lambda}$ lags behind $\bs w_\varepsilon$. This matters quantitatively: in the direction of $\nabla u$, the standard Newton tensor has eigenvalue $(p-1)|\nabla u|_\varepsilon^{p-2}$, i.e., it is $1/(p-1)$ times more compliant than the Picard operator exactly where the residual points in flat regions, whereas with a small lagged $\bs\lambda$ the lifted tensor stays close to the Picard operator.
The linearization in lifted Newton thus filters these fluctuations, %
an effect that decays as the iterates approach consistency and is therefore confined to the phase of the iteration before the Newton convergence basin is reached.

\begin{figure}
    \centering
    \includegraphics[width=0.95\linewidth]{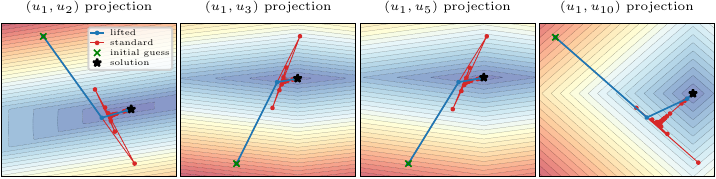}
    \caption{Lifted and standard Newton iterations for a finite difference discretization of the one-dimensional $p$-Laplace problem introduced in appendix \ref{app:one_dim_flux} with $p = 1.01$. Each panel shows the projection of iterates onto the two-dimensional subspace given by two nodal values.
    }
    \label{fig:fd_result}
\end{figure}

\noindent
\emph{An illustrative 1D $p$-Laplace problem.}
Appendix \ref{app:one_dim_flux} analyzes a one-dimensional problem, discretized with finite differences, whose solution is constructed to have a plateau. Expanding both linearizations in the small parameter $p-1$ shows that, on the plateau, standard and lifted Newton compute the \emph{same} increment of the flux $|\nabla u|_\varepsilon^{p-2}\nabla u$, namely the one that drives it to zero. They differ in how this flux increment is converted into an increment of $\nabla u$ (note that here $\nabla u=u'$): standard Newton divides by the tangent slope $(p-1)/|\nabla u|$ of the nearly flat constitutive function, so that every gradient on the plateau changes sign and grows by a factor $1/(p-1)$, whereas lifted Newton with a small lagged $\bs\lambda$ divides by the secant slope $1/|\nabla u|$ and sets the plateau gradients to zero up to $O(p-1)$ in one step. The update of the lifting variable then yields $\bs\lambda=O(p-1)$ on the plateau from the first iteration on. \Cref{fig:fd_result} shows both iterations for $p=1.01$ and a random initialization within an interval of length $0.1$ around the solution: standard Newton exhibits the predicted zig-zag and needs 27 iterations, while lifted Newton converges in 4. The first lifted step, for which $\bs\lambda_0=\bs 0$, is a Picard step and points in the same direction as the first standard Newton step.

\section{Numerical results: $p$-Laplace}
\label{sec:numerics-p-laplace}

\subsection{Implementation and comparisons}
\label{ssec:plaplace-setup}
We implement the lifted Newton method for the $p$-Laplace (and for the $p$-Stokes problem in \cref{sec:p-stokes} below) with the finite element (FE) library Firedrake \cite{FiredrakeUserManual}. The FE method represents the discrete solution $u_h$ in a finite element space $V_h$ and solves a weak formulation of the associated PDE. For the lifted Newton method, given the discrete variables $u_h$ and $\bs \lambda_h$, we seek an increment $\hat{u}_h\in \mathring{V}_h$, where $\mathring{V}_h$ denotes the subspace of $V_h$ that incorporates homogeneous Dirichlet boundary conditions. The increment $\hat{u}_h$ solves the linear system \eqref{eq:lifted-linear-sys} with the symmetrized diffusion tensor \eqref{eq:lifted-tensor-sym}. Writing $\bs w_{\varepsilon,h}$ and $\bs w_{\bs\lambda,h}$ for the discrete counterparts of the directions \eqref{eq:w-directions}, its weak form reads
\begin{equation}\label{eq:weak-lifted-linear-sys}
    \begin{split}
    \int_\Omega |\nabla u_h|_\varepsilon^{p-2}\Bigl(I + (p-2)\,\sym\bigl(\bs w_{\bs\lambda,h}\bs w_{\varepsilon,h}^\top\bigr)\Bigr)  \nabla \hat u_h \cdot \nabla v_h\,\mathrm{d}\bs x & \\ = \int_\Omega fv_h\,\mathrm{d}\bs x - \int_\Omega |\nabla u_h|_\varepsilon^{p-2}\nabla u_h \cdot \nabla v_h\,\mathrm{d}\bs x &\quad \forall v_h\in\mathring{V}_h.
    \end{split}
\end{equation}
Equation \eqref{eq:weak-lifted-linear-sys} can be implemented with a few lines of code, thanks to Firedrake's use of unified form language \cite{alneas2012}, a purely symbolic language for writing variational forms of PDEs, and PETSc \cite{petsc-user-ref}, a numerical linear algebra library. 

It remains to specify how the discrete lifting variable $\bs \lambda_h$ is represented. Note that the computation of $\bs \lambda_h$ is purely algebraic, in the sense that \eqref{eq:lambda-hat} does not involve derivatives of $\bs\lambda_h$. Moreover, the linear system \eqref{eq:weak-lifted-linear-sys} only requires the values of $\bs \lambda_h$ at quadrature points. For these reasons, we do not introduce $\bs\lambda_h$ as a FE function but simply store and update its values at quadrature points. In Firedrake, this is made possible with the so-called ``quadrature element'' functionality, as illustrated in the open-source implementation we provide.\footnote{The repository is available at \href{https://github.com/gonzalogddiego/lifted_newton}{https://github.com/gonzalogddiego/lifted\_newton}.}
Quadrature elements are special function spaces in Firedrake defined only on the points of a numerical quadrature rule. This contrasts with standard FE functions, characterized in terms of degrees of freedom (i.e.~linear forms defined on the FE space) and defined for all $\bs{x}\in\Omega$. Quadrature elements are a standard feature in many FE libraries; they are also available in FEniCS \cite{LoggEtal2012}, and an equivalent functionality can be found in MFEM \cite{AndersonEtAl2021}. 
We point out that the relationship between the variables $u$ and $
\bs \lambda$ differs from that found in mixed FE problems \cite{BrezziFortin91}. There, FE spaces are introduced for both variables, and these spaces must satisfy so-called inf-sup conditions to ensure the stability of the discrete problem. The explicit and purely algebraic nature of \eqref{eq:lambda-hat} dispenses with the need for such conditions, simplifying the construction of discrete lifted Newton schemes considerably.    
In the following, we report numerical comparisons of iterative schemes for solving the $p$-Laplace problem \eqref{eq:p-laplace} for different values of $p$ with $1<p<2$, for a problem in two and one in three dimensions. We study the influence of uniform and local mesh refinement and use linear and quadratic elements.
We present results for the Picard method \eqref{eq:picard-intro}, the classical Newton method \eqref{eq:p-Laplace-Newton}, and the lifted Newton iteration \eqref{eq:lifted-linear-sys}. Additionally, we compare the recently proposed Picard--Newton method \cite{PollockRebholzTuXiao25}, which alternates between Picard and Newton steps to improve the robustness of Newton's method. The method has been shown to be particularly useful for significantly nonlinear stationary Navier--Stokes problems. \Cref{tab:methods} summarizes the methods compared in this section for solving the $p$-Laplace equation, and some of these approaches are later applied to the $p$-Stokes problems in \cref{sec:p-Stokes-numerics}.
\begin{table}[bh]
    \centering
    \caption{Abbreviations used for the iterative methods compared in the numerical experiments.}
    \label{tab:methods}
    \begin{tabular}{lll}
        \toprule
        Name & Method & Iteration \\
        \midrule
        SN & Standard Newton  & \eqref{eq:p-Laplace-Newton} \\
        LN & Lifted Newton    & \eqref{eq:lifted-linear-sys}, \cref{ssec:lifted-algorithm} \\
        P  & Picard           & \eqref{eq:picard-intro} \\
        PN & Picard--Newton   & \cite{PollockRebholzTuXiao25} \\
        \bottomrule
        \vspace{-6ex}
    \end{tabular}
\end{table}

In all examples, we report the iteration counts, which are equivalent to the number of linear solves. This is justified since all methods require the solution of a symmetric, positive definite system for the degrees of freedom corresponding to the discretization of $u$. One step of the Picard--Newton method requires solving two linear systems, and this is counted as two iterations. Our implementation uses a direct solver for the symmetric linear systems arising in each iteration.
Unless otherwise specified, we use the parameters $\varepsilon=10^{-6}$ and $\beta=1-10^{-4}$. We use a slight modification of the Armijo line search and use
\begin{equation}\label{eq:Armijo-mod}
    J_\varepsilon(u_{h,k}+\alpha_k\hat u_{h,k})\le J_\varepsilon(u_{h,k})+\gamma\,\alpha_k\,J_\varepsilon'(u_{h,k})\hat u_{h,k}+\tau_{\,\rm ls},
\end{equation}
where $\tau_{\,\rm ls}$ at the order of machine precision (we use $\tau_{\,\rm ls}=10^{-15}$). Including this tolerance prevents the line search from failing because of rounding errors when the iterate is extremely close to the solution. This typically occurs in the final step, where the finite-precision value of $J_\varepsilon(u_{h,k})$ may not change, even though the residual can still be decreased by accepting a full update step.

In our numerical experiments we set $\gamma=0$ in \eqref{eq:Armijo-mod}. While this line-search variant does not enforce a prescribed fraction of the predicted decrease, the coercivity of $M_k$ ensures that $\hat u_{h,k}$ is a descent direction at each iteration, so the backtracking procedure is well defined. For several of the tests reported below, we also confirmed that taking $\gamma=10^{-4}$ yields the same number of iterations.
Iterations are terminated when the relative $H^{-1}$-norm of the residual reaches $10^{-6}$, or after 500 iterations, indicating that the method takes too long to converge.

\subsection{Two-dimensional $p$-Laplace equation example}
\label{ssec:plaplace-2d}

In this subsection, we use a two-dimensional example for the $p$-Laplace equation \eqref{eq:p-laplace} taken from \cite{Loisel20}. Let $\Omega=(0,1)^2$,  $f=0$, and we enforce a Dirichlet condition given by the characteristic function $g=1_X$, where
\begin{equation*}
  X=\bigl(\{0\}\times[0.25,0.75]\bigr)
    \cup\bigl([0.6,1]\times[0.25,1]\bigr).
\end{equation*}

\begin{figure}[bt]
    \centering
    \begin{tikzpicture}
    \def\ConvPanelWidth{0.242\textwidth}
    \definecolor{convBlue}{HTML}{0077BB}
    \definecolor{convOrange}{HTML}{EE7733}
    \definecolor{convTeal}{HTML}{009988}
    \definecolor{convPurple}{HTML}{AA3377}
    \pgfplotstableread{data/rel_grad_norms_plaplace_2d.txt}%
        \loiselRelativeGradientData
    \pgfplotstableread{data/step_lengths_plaplace_2d.txt}%
        \loiselStepSizeData
    \pgfplotsset{
        convergence axis/.style={
            width=\ConvPanelWidth,
            height=0.200\textwidth,
            scale only axis,
            axis line style={black!65, line width=0.35pt},
            tick style={black!65, line width=0.35pt},
            grid=major,
            major grid style={black!10, line width=0.25pt},
            xlabel={Iteration},
            title style={font=\small, yshift=-1pt},
            label style={font=\footnotesize},
            tick label style={font=\footnotesize},
            scaled ticks=false,
            unbounded coords=jump,
            filter discard warning=false,
        },
        convergence curve/.style={
            solid,
            no marks,
            line width=0.85pt,
            line cap=round,
            line join=round,
        },
        standard curve/.style={
            convergence curve,
            color=convBlue,
        },
        lifted curve/.style={
            convergence curve,
            color=convOrange,
        },
        picard curve/.style={
            convergence curve,
            color=convTeal,
        },
        picard newton curve/.style={
            convergence curve,
            color=convPurple,
        },
        convergence legend/.style={
            draw=black!25,
            fill=white,
            rounded corners=1pt,
            font=\footnotesize,
            cells={anchor=west},
            legend image code/.code={
                \draw[
                    /pgfplots/mesh=false,
                    bar width=3pt,
                    bar shift=0pt,
                    mark repeat=2,
                    mark phase=2,
                    ##1
                ] plot coordinates {
                    (0cm,0cm) (0.125cm,0cm) (0.25cm,0cm)
                };
            },
            column sep=0.25em,
            row sep=-1.5pt,
            inner xsep=1.5pt,
            inner ysep=1.5pt,
        },
    }
        \begin{groupplot}[
            group style={
                group size=3 by 2,
                horizontal sep=0.040\textwidth,
                vertical sep=0.020\textwidth,
            },
            convergence axis,
        ]
        \nextgroupplot[
            xlabel={},
            xticklabels=\empty,
            ylabel={Relative residual norm},
            xmin=0, xmax=18.5,
            xtick={0,5,10,15},
            ymode=log,
            ymin=5e-8, ymax=2,
            ytick={1e-7,1e-5,1e-3,1e-1,1e0},
            legend to name=loiselLegend,
            legend columns=1,
            legend style={convergence legend},
        ]
        \addplot[standard curve]
            table[x=iteration, y=standard_p160]
            {\loiselRelativeGradientData};
        \addlegendentry{SN}
        \addplot[lifted curve]
            table[x=iteration, y=lifted_p160]
            {\loiselRelativeGradientData};
        \addlegendentry{LN}
        \addplot[picard curve]
            table[x=iteration, y=picard_p160]
            {\loiselRelativeGradientData};
        \addlegendentry{P}
        \addplot[picard newton curve]
            table[x=iteration, y=pnls_p160]
            {\loiselRelativeGradientData};
        \addlegendentry{PN}

        \nextgroupplot[
            xlabel={},
            xticklabels=\empty,
            xmin=0, xmax=45,
            xtick={0,10,20,30,40},
            ymode=log,
            ymin=5e-8, ymax=2,
            ytick={1e-7,1e-5,1e-3,1e-1,1e0},
            yticklabels=\empty,
        ]
        \addplot[standard curve]
            table[x=iteration, y=standard_p110]
            {\loiselRelativeGradientData};
        \addplot[lifted curve]
            table[x=iteration, y=lifted_p110]
            {\loiselRelativeGradientData};
        \addplot[picard curve]
            table[x=iteration, y=picard_p110]
            {\loiselRelativeGradientData};
        \addplot[picard newton curve]
            table[x=iteration, y=pnls_p110]
            {\loiselRelativeGradientData};

        \nextgroupplot[
            xlabel={},
            xticklabels=\empty,
            xmin=0, xmax=100,
            xtick={0,25,50,75,100},
            ymode=log,
            ymin=5e-8, ymax=2,
            ytick={1e-7,1e-5,1e-3,1e-1,1e0},
            yticklabels=\empty,
        ]
        \addplot[standard curve]
            table[x=iteration, y=standard_p101]
            {\loiselRelativeGradientData};
        \addplot[lifted curve]
            table[x=iteration, y=lifted_p101]
            {\loiselRelativeGradientData};
        \addplot[picard curve]
            table[x=iteration, y=picard_p101]
            {\loiselRelativeGradientData};
        \addplot[picard newton curve]
            table[x=iteration, y=pnls_p101]
            {\loiselRelativeGradientData};

        \nextgroupplot[
            ylabel={Step size},
            xmin=0, xmax=18.5,
            xtick={0,5,10,15},
            ymode=log,
            log basis y=2,
            ymin=0.012, ymax=1.2,
            ytick={0.015625,0.0625,0.25,1},
            yticklabels={$1/64$,$1/16$,$1/4$,$1$},
        ]
        \addplot[standard curve]
            table[x=iteration,y=standard_p160]
            {\loiselStepSizeData};
        \addplot[lifted curve]
            table[x=iteration,y=lifted_p160]
            {\loiselStepSizeData};
        \addplot[picard newton curve]
            table[x expr=\thisrow{iteration}*2,y=pnls_newton_p160]
            {\loiselStepSizeData};

        \nextgroupplot[
            xmin=0, xmax=45,
            xtick={0,10,20,30,40},
            ymode=log,
            log basis y=2,
            ymin=0.012, ymax=1.2,
            ytick={0.015625,0.0625,0.25,1},
            yticklabels=\empty,
        ]
        \addplot[standard curve]
            table[x=iteration,y=standard_p110]
            {\loiselStepSizeData};
        \addplot[lifted curve]
            table[x=iteration,y=lifted_p110]
            {\loiselStepSizeData};
        \addplot[picard newton curve]
            table[x expr=\thisrow{iteration}*2,y=pnls_newton_p110]
            {\loiselStepSizeData};

        \nextgroupplot[
            xmin=0, xmax=100,
            xtick={0,25,50,75,100},
            ymode=log,
            log basis y=2,
            ymin=0.012, ymax=1.2,
            ytick={0.015625,0.0625,0.25,1},
            yticklabels=\empty,
        ]
        \addplot[standard curve]
            table[x=iteration,y=standard_p101]
            {\loiselStepSizeData};
        \addplot[lifted curve]
            table[x=iteration,y=lifted_p101]
            {\loiselStepSizeData};
        \addplot[picard newton curve]
            table[x expr=\thisrow{iteration}*2,y=pnls_newton_p101]
            {\loiselStepSizeData};
        \end{groupplot}

        \node[anchor=south, inner sep=0pt] (sol160)
            at ([yshift=0.025\textwidth]group c1r1.north)
            {\includegraphics[
                width=\ConvPanelWidth,
                trim=8.3bp 8.3bp 8.3bp 8.3bp,
                clip
            ]{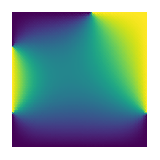}};
        \node[anchor=south, inner sep=0pt] (sol110)
            at ([yshift=0.025\textwidth]group c2r1.north)
            {\includegraphics[
                width=\ConvPanelWidth,
                trim=8.3bp 8.3bp 8.3bp 8.3bp,
                clip
            ]{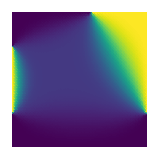}};
        \node[anchor=south, inner sep=0pt] (sol101)
            at ([yshift=0.025\textwidth]group c3r1.north)
            {\includegraphics[
                width=\ConvPanelWidth,
                trim=8.3bp 8.3bp 8.3bp 8.3bp,
                clip
            ]{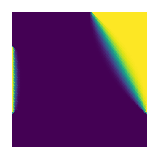}};
        \node[anchor=south, font=\small, inner sep=0pt]
            at ([yshift=3pt]sol160.north) {$p=1.6$};
        \node[anchor=south, font=\small, inner sep=0pt]
            at ([yshift=3pt]sol110.north) {$p=1.1$};
        \node[anchor=south, font=\small, inner sep=0pt]
            at ([yshift=3pt]sol101.north) {$p=1.01$};
        \node[anchor=west, inner sep=0pt] (loiselColorbar)
            at ([xshift=0.010\textwidth]sol101.east)
            {\includegraphics[
                height=\ConvPanelWidth,
                trim=2.9bp 6.5bp 15.65bp 5.1bp,
                clip
            ]{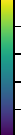}};
        \foreach \position/\value in {
            0/{$0.0$},.2/{$0.2$},.4/{$0.4$},
            .6/{$0.6$},.8/{$0.8$},1/{$1.0$}
        }{
            \node[anchor=west, font=\footnotesize, inner sep=0pt]
                at ([xshift=3pt]$(loiselColorbar.south east)!\position!
                    (loiselColorbar.north east)$) {\value};
        }

        \node[anchor=south east, inner sep=0pt]
            at ([xshift=-3pt,yshift=3pt]group c1r2.south east)
            {\pgfplotslegendfromname{loiselLegend}};
    \end{tikzpicture}
    \caption{Solutions and convergence behavior for the two-dimensional $p$-Laplace example on the $N=50$ grid ($10{,}000$ triangles). The top row shows the solutions for $p=1.6$, $p=1.1$, and $p=1.01$. The middle row shows the relative $H^{-1}$-norm of the residual, i.e., $\|r_{h,k}\|_{H^{-1}}/\|r_{h,0}\|_{H^{-1}}$
    with a shared vertical scale. Normalization is performed separately for each $p$; all methods share the same initial iterate within a column, so every curve starts at one. The bottom row shows the accepted line-search step sizes on a shared base-2 vertical scale. Picard is omitted from the bottom row because it does not use line search; for Picard--Newton, the step size of the Newton substep is shown. Histories extending beyond a panel's horizontal range are truncated; complete iteration counts are given in \cref{tab:p-Lapl-2d}.}
    \label{fig:p-Lapl-2d}
\end{figure}
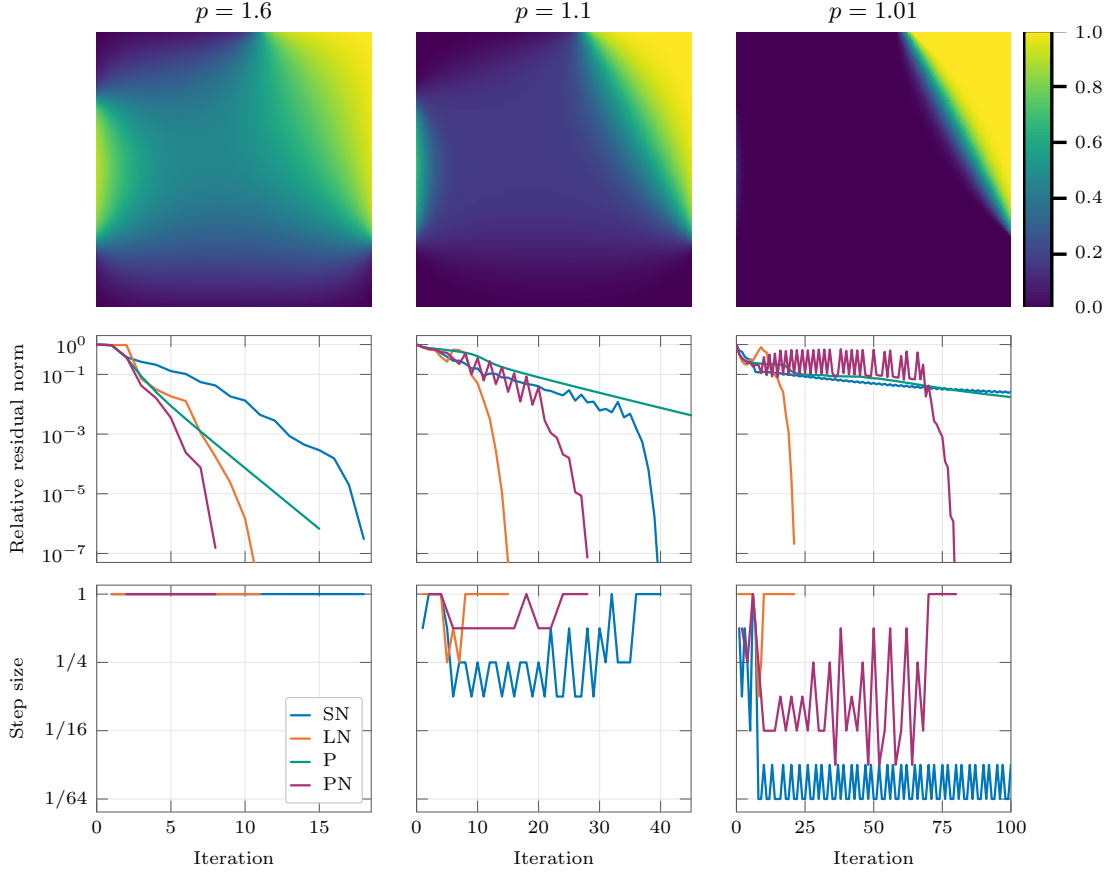

We initialize the primal solution by $u_0=0$ and impose the 
Dirichlet boundary condition strongly. The lifting variable is initialized as $\bs\lambda_0=\bs 0$, resulting in the first iteration being a Picard step as discussed at the end of \cref{ssec:lifted-algorithm}. Other initialization strategies we tried resulted in the same, or nearly the same, convergence behavior.
We use continuous piecewise-linear elements on uniform $N\times N$ regular grids with $N\in\{25,50,100\}$ and crossed diagonals. Each square of the Cartesian partition is split into four triangles, so the three meshes contain $4N^2=2{,}500$, $10{,}000$, and $40{,}000$ triangles, respectively.

In \cref{fig:p-Lapl-2d}, we show the solutions for several values of $p$ together with the performance of the different solvers on the $N=50$ grid. One observes that, as $p$ tends to $1$, the transition layer becomes increasingly sharp, which makes the problem more difficult to solve. The iteration counts on all three grids are reported in \cref{tab:p-Lapl-2d}.
When $p$ approaches $1$, the iteration counts rise steeply for all methods, with the exception of LN, whose counts grow only mildly on every grid. The linear convergence of P degrades as $p\to 1$, so it fails to achieve the prescribed relative tolerance for $p=1.01$ on all three grids. The hybrid PN offers a substantial improvement over both P and SN, but for $p\approx 1$ it is still less efficient than LN.
The line search step lengths in \cref{fig:p-Lapl-2d} explain the different convergence behaviors between SN and LN. For $p=1.01$, SN spends a long initial phase accepting strongly damped steps; many accepted steps are approximately $1/32$. On the other hand, LN takes only a few damped steps before accepting full steps, which accelerates the convergence.

In \cref{tab:p-Lapl-2d} we summarize the performance of the different methods over a range of $p$ values; each entry lists the counts for $N=25$, $50$, and $100$. For $p=1.6$, the spread of required iterations is modest on every grid, whereas the discrepancies between them increase markedly as $p$ approaches $1$. It can also be seen that LN is substantially less sensitive to mesh refinement.

\begin{table}[t]
    \centering
    \caption{Number of iterations needed for convergence by different methods applied to the two-dimensional $p$-Laplace problem. Each table entry gives the iteration counts for the $N=25$, $N=50$, and $N=100$ grids, reported as X/X/X; the corresponding meshes have $2{,}500$, $10{,}000$, and $40{,}000$ triangles. The label ``n.c.'' indicates that the specified tolerance was not achieved within 500 iterations.}  \label{tab:p-Lapl-2d}
    \begin{tabular}{lrrrrr}
        \toprule
        Method & $p=1.6$ & $p=1.3$ & $p=1.1$ & $p=1.05$ & $p=1.01$ \\
        \midrule
        SN & 15/18/22 & 23/29/33 & 34/40/57 & 79/87/105 & 283/329/358 \\
        LN & 9/11/11 & 10/10/10 & 18/15/16 & 18/19/20 & 23/21/24 \\
        P & 15/15/15 & 38/38/38 & 112/119/123 & 203/218/232 & n.c./n.c./n.c. \\
        PN & 8/8/8 & 12/14/18 & 24/28/30 & 32/38/46 & 62/80/92 \\
        \bottomrule
    \end{tabular}
\end{table}

Finally, in \cref{tab:p-Lapl-2d-epsilon}, we show the number of iterations required for different choices of regularization parameters $\varepsilon$. For $p\approx 1$, a smaller $\varepsilon$ makes the problem more singular and generally increases the number of iterations required for convergence. In particular, for $\varepsilon=10^{-10}$ and $p=1.01$, the residuals of both methods fail to reach the prescribed tolerance, likely due to roundoff errors for these extreme parameter values. This shows the necessity of  the $\varepsilon$-regularization for solver robustness.

\begin{table}[ht]
    \centering
    \caption{Comparison of the number of iterations required by SN and LN for various regularization parameters $\varepsilon$ in the two-dimensional $p$-Laplace example for $N=50$.
    As done throughout, the iteration is terminated when the relative residual reaches $10^{-6}$; the entries marked with $\dagger$ correspond to a relative residual of $10^{-5}$. For these tests,  the residual does not decrease to $10^{-6}$, likely due to rounding errors due to extreme values of $p$ and $\varepsilon$.}  \label{tab:p-Lapl-2d-epsilon}
    \begin{tabular}{lcrrrrr}
        \toprule
        Method & $\varepsilon$ & $p=1.6$ & $p=1.3$ & $p=1.1$ & $p=1.05$ & $p=1.01$ \\
        \midrule
        SN & $10^{-4}$  & 19 & 27 & 38 & 65  & 206 \\
        SN & $10^{-6}$  & 18 & 29 & 40 & 87  & 329 \\
        SN & $10^{-8}$  & 18 & 21 & 41 & 112 & 443 \\
        SN & $10^{-10}$ & 17 & 22 & 42 & 127 & n.c.$^\dagger$ \\
        \addlinespace
        LN & $10^{-4}$  & 8  & 11 & 12 & 15 & 18 \\
        LN & $10^{-6}$  & 11 & 10 & 15 & 19 & 21 \\
        LN & $10^{-8}$  & 9  & 11 & 19 & 21 & 31 \\
        LN & $10^{-10}$ & 9  & 11 & 21 & 23 & 30$^\dagger$ \\
        \bottomrule
    \end{tabular}
\end{table}

\subsection{Three-dimensional $p$-Laplace equation example}
\label{ssec:plaplace-3d}

In this subsection, we consider a three-dimensional $p$-Laplace test problem and compare some of the solution methods from \cref{tab:methods} for two finite element spaces with different polynomial orders on a locally refined mesh.  We use $\Omega=(0,1)^3$, again set $f=0$ and use the Dirichlet boundary data $g=1_{X}$, where
\begin{equation*}
    X=\left\{(x,y,z)\in [0,1]^3: x+y+z\geq \tfrac{5}{2}\right\}.
\end{equation*}
Similar to the previous two-dimensional example, this corner-cutting problem produces a sharp transition close to
$
    \Gamma:=\bigl\{(x,y,z)\in[0,1]^3:x+y+z=\frac{5}{2}\bigr\}
$,
especially for $p\approx1$. To properly refine this transition zone, we use a locally refined finite element mesh. We start from a uniform $5\times5\times5$ Cartesian partition of the cube $\Omega$ and divide each cube into $6$ tetrahedral cells, resulting in a coarse mesh with $750$ cells. We then perform three successive local refinement passes. At each pass, a tetrahedron is marked when the range of $x+y+z$ over its vertices intersects the plane-centered slab of perpendicular half-width $0.01$; in other words, we only refine every cell that is cut by $\Gamma$ since the slab is very thin. The resulting mesh used in the computations contains $N=16{,}848$ tetrahedra, of which $3{,}612$ intersect $\Gamma$. Plane-intersecting cells have a diameter of $h_K=\sqrt{3}/40\approx0.0433$. On this locally refined mesh, we use both continuous piecewise-linear (CG1) and piecewise-quadratic (CG2) finite element spaces, whose degrees of freedom are $3{,}352$ and $24{,}406$, respectively. The mesh is visualized in the top left of \cref{fig:p-Lapl-3d}.

We use the same initializations %
as in \cref{ssec:plaplace-2d}. The CG2~solutions and the convergence behavior for $p\in\{1.6,1.1,1.01\}$ are shown in \cref{fig:p-Lapl-3d}, while \cref{tab:p-Lapl-3d} reports the iteration counts of all methods for both CG1 and CG2. As in the two-dimensional example, LN converges in substantially fewer iterations than SN, and the difference becomes more significant as $p$ approaches $1$.

The results in \cref{tab:p-Lapl-3d} also show that LN is robust to the polynomial degree: its CG1 and CG2 iteration counts differ by at most three. For $p=1.01$, lifted Newton requires $21$ linear solves with CG1 and $24$ with CG2, compared to $257$ and $242$, respectively, for SN. P does not reach the prescribed tolerance within $500$ iterations for either discretization, whereas PN converges in $50$ linear solves with CG1 and $44$ with CG2.

\begin{figure}[t]
    \centering
    \begin{tikzpicture}
    \def\ConvThreeDPanelWidth{0.242\textwidth}
    \definecolor{convThreeDBlue}{HTML}{0077BB}
    \definecolor{convThreeDOrange}{HTML}{EE7733}
    \definecolor{convThreeDTeal}{HTML}{009988}
    \definecolor{convThreeDPurple}{HTML}{AA3377}
    \pgfplotstableread{data/rel_grad_norms_plaplace_3d.txt}%
        \threeDGradientData
    \pgfplotstableread{data/step_lengths_plaplace_3d.txt}%
        \threeDStepSizeData
    \pgfplotsset{
        three d convergence axis/.style={
            width=\ConvThreeDPanelWidth,
            height=0.200\textwidth,
            scale only axis,
            axis line style={black!65, line width=0.35pt},
            tick style={black!65, line width=0.35pt},
            grid=major,
            major grid style={black!10, line width=0.25pt},
            xlabel={Iteration},
            label style={font=\footnotesize},
            tick label style={font=\footnotesize},
            scaled ticks=false,
            unbounded coords=jump,
            filter discard warning=false,
        },
        three d convergence curve/.style={
            solid,
            no marks,
            line width=0.85pt,
            line cap=round,
            line join=round,
        },
        three d standard curve/.style={
            three d convergence curve,
            color=convThreeDBlue,
        },
        three d lifted curve/.style={
            three d convergence curve,
            color=convThreeDOrange,
        },
        three d picard curve/.style={
            three d convergence curve,
            color=convThreeDTeal,
        },
        three d picard newton curve/.style={
            three d convergence curve,
            color=convThreeDPurple,
        },
        three d convergence legend/.style={
            draw=black!25,
            fill=white,
            rounded corners=1pt,
            font=\footnotesize,
            cells={anchor=west},
            legend image code/.code={
                \draw[
                    /pgfplots/mesh=false,
                    bar width=3pt,
                    bar shift=0pt,
                    mark repeat=2,
                    mark phase=2,
                    ##1
                ] plot coordinates {
                    (0cm,0cm) (0.125cm,0cm) (0.25cm,0cm)
                };
            },
            column sep=0.25em,
            row sep=-1.5pt,
            inner xsep=1.5pt,
            inner ysep=1.5pt,
        },
    }
        \begin{groupplot}[
            group style={
                group size=3 by 2,
                horizontal sep=0.040\textwidth,
                vertical sep=0.020\textwidth,
            },
            three d convergence axis,
        ]
        \nextgroupplot[
            xlabel={},
            xticklabels=\empty,
            ylabel={Relative residual norm},
            xmin=0, xmax=16.5,
            xtick={0,5,10,15},
            ymode=log,
            ymin=5e-8, ymax=2,
            ytick={1e-7,1e-5,1e-3,1e-1,1e0},
            legend to name=threeDLegend,
            legend columns=1,
            legend style={three d convergence legend},
        ]
        \addplot[three d standard curve]
            table[x=iteration, y=standard_p160]
            {\threeDGradientData};
        \addlegendentry{SN}
        \addplot[three d lifted curve]
            table[x=iteration, y=lifted_p160]
            {\threeDGradientData};
        \addlegendentry{LN}
        \addplot[three d picard curve]
            table[x=iteration, y=picard_p160]
            {\threeDGradientData};
        \addlegendentry{P}
        \addplot[three d picard newton curve]
            table[x=iteration, y=pnls_p160]
            {\threeDGradientData};
        \addlegendentry{PN}

        \nextgroupplot[
            xlabel={},
            xticklabels=\empty,
            xmin=0, xmax=100,
            xtick={0,25,50,75,100},
            ymode=log,
            ymin=5e-8, ymax=2,
            ytick={1e-7,1e-5,1e-3,1e-1,1e0},
            yticklabels=\empty,
        ]
        \addplot[three d standard curve]
            table[x=iteration, y=standard_p110]
            {\threeDGradientData};
        \addplot[three d lifted curve]
            table[x=iteration, y=lifted_p110]
            {\threeDGradientData};
        \addplot[three d picard curve]
            table[x=iteration, y=picard_p110]
            {\threeDGradientData};
        \addplot[three d picard newton curve]
            table[x=iteration, y=pnls_p110]
            {\threeDGradientData};

        \nextgroupplot[
            xlabel={},
            xticklabels=\empty,
            xmin=0, xmax=60,
            xtick={0,15,30,45,60},
            ymode=log,
            ymin=5e-8, ymax=2,
            ytick={1e-7,1e-5,1e-3,1e-1,1e0},
            yticklabels=\empty,
        ]
        \addplot[three d standard curve]
            table[x=iteration, y=standard_p101]
            {\threeDGradientData};
        \addplot[three d lifted curve]
            table[x=iteration, y=lifted_p101]
            {\threeDGradientData};
        \addplot[three d picard curve]
            table[x=iteration, y=picard_p101]
            {\threeDGradientData};
        \addplot[three d picard newton curve]
            table[x=iteration, y=pnls_p101]
            {\threeDGradientData};

        \nextgroupplot[
            ylabel={Step size},
            xmin=0, xmax=16.5,
            xtick={0,5,10,15},
            ymode=log,
            log basis y=2,
            ymin=0.012, ymax=1.2,
            ytick={0.015625,0.0625,0.25,1},
            yticklabels={$1/64$,$1/16$,$1/4$,$1$},
        ]
        \addplot[three d standard curve]
            table[x=iteration,y=standard_p160]
            {\threeDStepSizeData};
        \addplot[three d lifted curve]
            table[x=iteration,y=lifted_p160]
            {\threeDStepSizeData};
        \addplot[three d picard newton curve]
            table[x expr=\thisrow{iteration}*2,y=pnls_newton_p160]
            {\threeDStepSizeData};

        \nextgroupplot[
            xmin=0, xmax=100,
            xtick={0,25,50,75,100},
            ymode=log,
            log basis y=2,
            ymin=0.012, ymax=1.2,
            ytick={0.015625,0.0625,0.25,1},
            yticklabels=\empty,
        ]
        \addplot[three d standard curve]
            table[x=iteration,y=standard_p110]
            {\threeDStepSizeData};
        \addplot[three d lifted curve]
            table[x=iteration,y=lifted_p110]
            {\threeDStepSizeData};
        \addplot[three d picard newton curve]
            table[x expr=\thisrow{iteration}*2,y=pnls_newton_p110]
            {\threeDStepSizeData};

        \nextgroupplot[
            xmin=0, xmax=60,
            xtick={0,15,30,45,60},
            ymode=log,
            log basis y=2,
            ymin=0.012, ymax=1.2,
            ytick={0.015625,0.0625,0.25,1},
            yticklabels=\empty,
        ]
        \addplot[three d standard curve]
            table[x=iteration,y=standard_p101]
            {\threeDStepSizeData};
        \addplot[three d lifted curve]
            table[x=iteration,y=lifted_p101]
            {\threeDStepSizeData};
        \addplot[three d picard newton curve]
            table[x expr=\thisrow{iteration}*2,y=pnls_newton_p101]
            {\threeDStepSizeData};
        \end{groupplot}

        \node[anchor=south, inner sep=0pt, outer sep=0pt] (threeDSol160)
            at ([yshift=0.025\textwidth]group c1r1.north)
            {\includegraphics[width=\ConvThreeDPanelWidth]
                {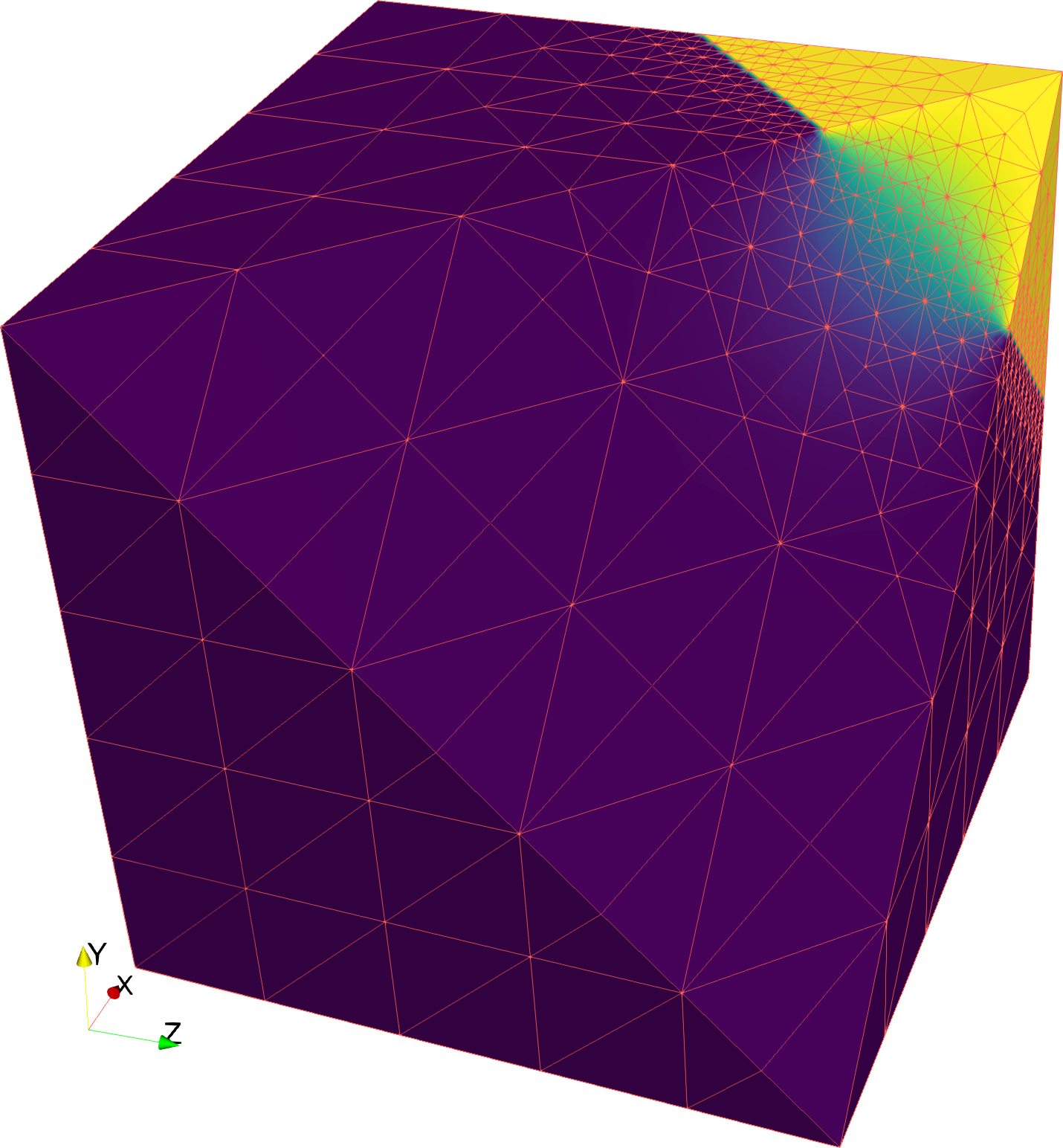}};
        \node[anchor=south, inner sep=0pt, outer sep=0pt] (threeDSol110)
            at ([yshift=0.025\textwidth]group c2r1.north)
            {\includegraphics[width=\ConvThreeDPanelWidth]
                {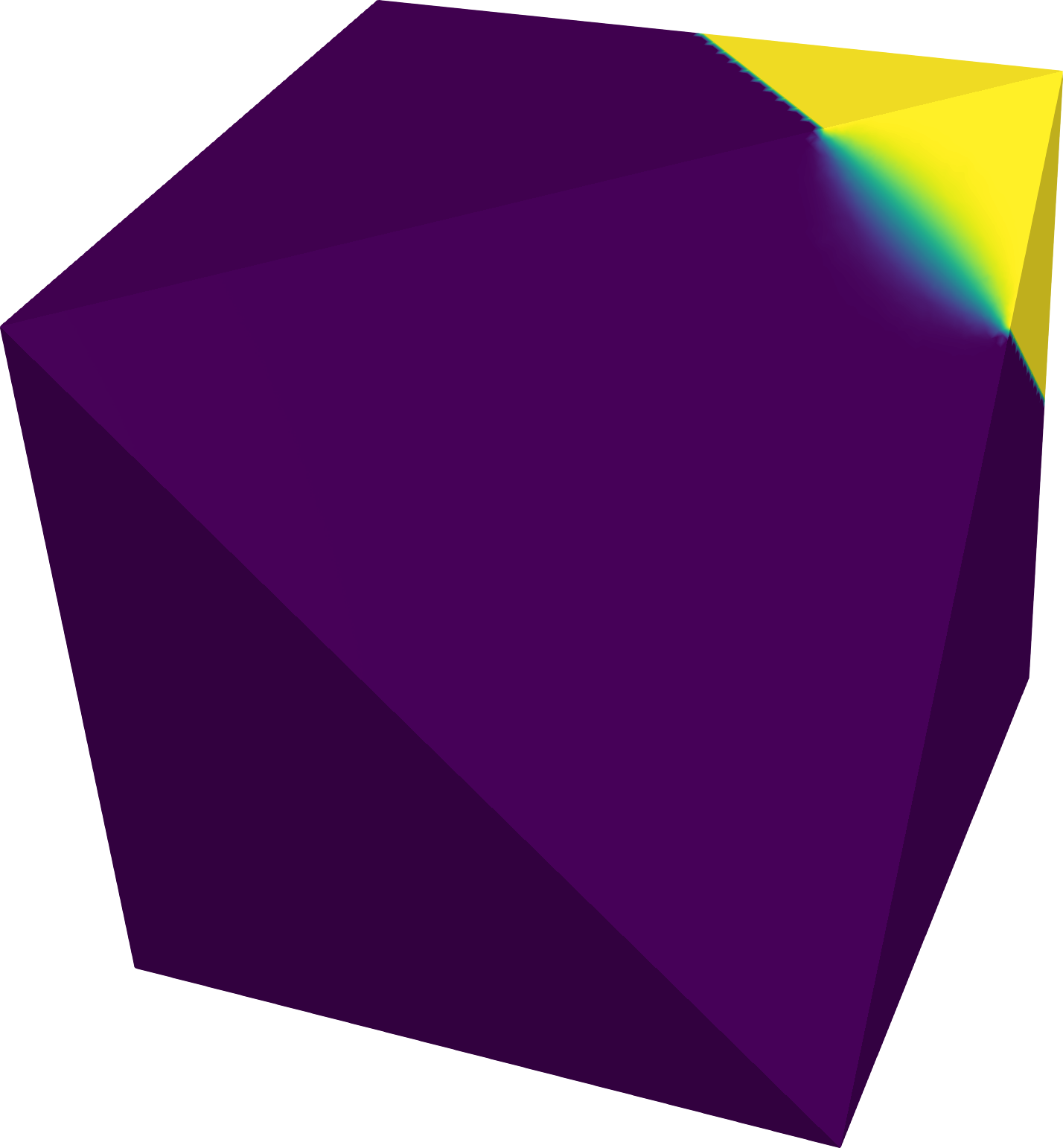}};
        \node[anchor=south, inner sep=0pt, outer sep=0pt] (threeDSol101)
            at ([yshift=0.025\textwidth]group c3r1.north)
            {\includegraphics[width=\ConvThreeDPanelWidth]
                {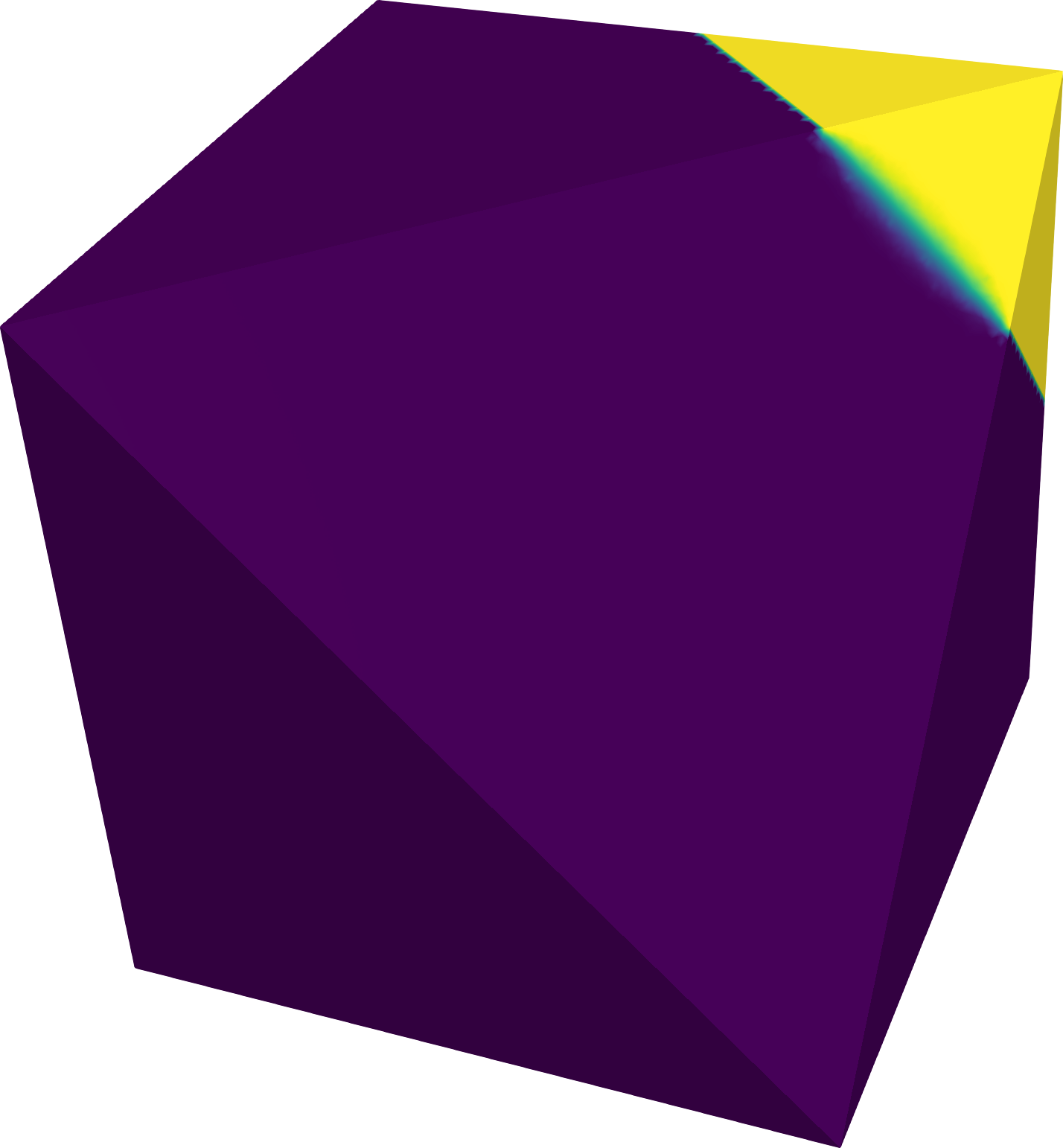}};
        \node[anchor=south, font=\small, inner sep=0pt]
            at ([yshift=3pt]threeDSol160.north) {$p=1.6$};
        \node[anchor=south, font=\small, inner sep=0pt]
            at ([yshift=3pt]threeDSol110.north) {$p=1.1$};
        \node[anchor=south, font=\small, inner sep=0pt]
            at ([yshift=3pt]threeDSol101.north) {$p=1.01$};
        \node[anchor=west, inner sep=0pt] (threeDColorbar)
            at ([xshift=0.010\textwidth]threeDSol101.east)
            {\includegraphics[
                height=\ConvThreeDPanelWidth,
                trim=2.9bp 6.5bp 15.65bp 5.1bp,
                clip
            ]{figs/plaplace_colorbar.pdf}};
        \foreach \position/\value in {
            0/{$0.0$},.2/{$0.2$},.4/{$0.4$},
            .6/{$0.6$},.8/{$0.8$},1/{$1.0$}
        }{
            \node[anchor=west, font=\footnotesize, inner sep=0pt]
                at ([xshift=3pt]$(threeDColorbar.south east)!\position!
                    (threeDColorbar.north east)$) {\value};
        }

        \node[anchor=south east, inner sep=0pt]
            at ([xshift=-3pt,yshift=3pt]group c1r2.south east)
            {\pgfplotslegendfromname{threeDLegend}};
    \end{tikzpicture}
    \caption{Solution and convergence behavior for the three-dimensional $p$-Laplace example using CG2 elements on the locally refined mesh with $16{,}848$ tetrahedra. The top row shows part of the solution fields computed with lifted Newton for $p=1.6$ (with the mesh visualized), $p=1.1$, and $p=1.01$. The middle row compares the relative $H^{-1}$-norm of the residual $\|r_{h,k}\|_{H^{-1}}/\|r_{h,0}\|_{H^{-1}}$ for all four methods on a shared vertical scale. The bottom row shows the accepted line-search step sizes on a shared base-2 vertical scale. Picard is omitted from the bottom row because it does not use line search; for PN, the step size of the Newton substep is shown. Histories extending beyond a panel's horizontal range are truncated; complete iteration counts are given in \cref{tab:p-Lapl-3d}.}
    \label{fig:p-Lapl-3d}
\end{figure}

\begin{table}[t]
    \centering
    \caption{Iterations required for convergence of the various methods on the locally refined mesh with $N=16{,}848$ tetrahedra, using linear (CG1) and quadratic (CG2) finite elements.
    The ``n.c.'' means that the prescribed tolerance was not reached within $500$ iterations.}
    \label{tab:p-Lapl-3d}

    \begin{tabular}{llrrr}
        \toprule
        Discretization & Method & $p=1.6$ & $p=1.1$ & $p=1.01$ \\
        \midrule
        CG1 & SN  & 10 & 43 & 257 \\
            & LN   & 7  & 15 & 21  \\
            & P          & 15 & 102 & n.c. \\
            & PN  & 8  & 30 & 50  \\
        \addlinespace
        CG2 & SN & 14 & 47  & 242 \\
            & LN   & 7  & 15  & 24  \\
            & P          & 15 & 105 & n.c. \\
            & PN  & 8  & 26  & 44  \\
        \bottomrule
    \end{tabular}
\end{table}

\section{Generalization to $p$-Stokes problems}
\label{sec:p-stokes}
We now turn to the $p$-Stokes problem, a generalization of the $p$-Laplacian of \cref{sec:p-laplacian} obtained by coupling the same type of power-law nonlinearity to an incompressibility constraint \cite{BulicekGwiazdaMalekEtAl12, CuffeyPaterson10}.
The $p$-Stokes system describes the slow, creeping flow of a generalized Newtonian (power-law) fluid. It differs from the $p$-Laplace problem in that the unknown is a divergence-free vector field. For $1<p< 2$, such fluids are said to be shear-thinning. The resulting near-singular viscosity makes Newton's method slow and fragile, with iteration counts that grow with 
mesh refinement and as $p$ gets close to $1$
\cite{schmidt-pStokes-global,MM-pNavierStokes}. We next specify the $p$-Stokes problem and then follow the same lift--transform--linearize--eliminate procedure as in \cref{ssec:newton}.

Let $\Omega\subset\mathbb{R}^{d}$, $d\in\{2,3\}$, be a bounded Lipschitz domain with boundary $\partial\Omega$.
In the $p$-Stokes problem, we seek a velocity field
$\bs{u}:\Omega\to\mathbb{R}^{d}$ and a pressure
$\pi:\Omega\to\mathbb{R}$ satisfying
\begin{subequations}
\label{eq:p-stokes-strong}
\begin{align}
-\nabla\!\cdot\!\bs{S}\!\left(\bs{D}\bs{u}\right) + \nabla\pi
   &= \bs{f}            && \text{in }\Omega,
   \label{eq:p-stokes-mom}\\
\nabla\!\cdot\!\bs{u}
   &= 0                 && \text{in }\Omega,
   \label{eq:p-stokes-mass}\\
\bs{u} &= \bs{u}_{\mathrm D} && \text{on }\partial\Omega,%
\end{align}
\end{subequations}
where
$\bs{D}\bs{u}:=\tfrac{1}{2}\bigl(\nabla\bs{u}+\nabla\bs{u}^\top\bigr)$
is the symmetric velocity gradient (strain rate tensor) and
$\bs{f}$, $\bs{u}_{\mathrm D}$ are prescribed
data. Clearly, more complicated boundary conditions are possible, but we restrict ourselves to Dirichlet conditions for simplicity.
The deviatoric extra-stress tensor $\bs{S}$ is given by
a power-law constitutive relation
\begin{equation*}
\bs{S}(\bs{D}\bs{u})\;=\;2\,\nu_0 |\bs{D}\bs{u}|^{p-2}\bs{D}\bs{u},
\qquad
\qquad p\in(1,\infty),
\end{equation*}
with $\nu_{0}>0$ and $|\bs{D}\bs{u}|^{2}:=\bs{D}\bs{u}\!:\!\bs{D}\bs{u}$ is the second invariant of the strain rate tensor.

As in \cref{sec:p-laplacian}, we introduce a small parameter $\varepsilon>0$ to regularize the second invariant of the strain-rate tensor, thereby avoiding the singularity for $p<2$. Specifically, we set
$|\bs{D}\bs{u}|_\varepsilon := \sqrt{|\bs{D}\bs{u}|^2+\varepsilon^2}$, and define the associated regularized stress by $\bs{S}_\varepsilon(\bs{D}\bs{u}) := 2\nu_0|\bs{D}\bs{u}|_\varepsilon^{p-2}\bs{D}\bs{u}$. Analogously, we introduce the tensor-valued variable $\bs{\Lambda}(\bs{x}) := \bs{S}_\varepsilon(\bs{D}\bs{u})(\bs{x})$ and consider the following regularized, lifted, and transformed system for $(\bs{u}, \pi, \bs{\Lambda})$:
\begin{subequations}
\label{eq:p-stokes-lifted}
\begin{alignat}{2}
-\nabla\!\cdot\!\bs{\Lambda} + \nabla\pi
   &= \bs{f}         \quad   && \text{in }\Omega,
   \label{eq:p-stokes-lifted-mom}\\
\nabla\!\cdot\!\bs{u}
   &= 0                 && \text{in }\Omega,
   \label{eq:p-stokes-lifted-mass}\\
   \bs{\Lambda} |\bs{D}\bs{u}|_\varepsilon^{2-p} - 2\nu_0\bs{D}\bs{u} &= 0 && \text{in }\Omega,
\end{alignat}
\end{subequations}
together with boundary conditions as in \eqref{eq:p-stokes-strong}.
Following the same procedure as in \cref{ssec:newton}, we linearize \eqref{eq:p-stokes-lifted} around $(\bs{u}, \pi, \bs{\Lambda})$ in the direction $(\bs{\hat{u}}, \hat{\pi}, \bs{\hat{\Lambda}})$. Then, by isolating the increment $\bs{\hat{\Lambda}}$ and eliminating it from the system, we arrive at the following linearized system for the velocity/pressure increment $(\bs{\hat u}, \hat\pi)$:
\begin{subequations} \label{eq:p-stokes-tangent}
\begin{alignat}{1}
\label{eq:p-stokes-tangent1}
- \nabla\cdot \left( |\bs{D}\bs{u}|_\varepsilon^{p-2}\! \left(2\nu_0\bs{D}\bs{\hat{u}} \!+\! (p\!-\!2) \left({\bs{W}_{\!\bs\Lambda}} \otimes \bs{W}_{\!\varepsilon} \right):\bs{D}\hat{\bs{u}}\right) \right) + \nabla\hat{\pi} &= \bs{f} + \nabla\!\cdot\!\bs{S}_\varepsilon\!\left(\bs{D}{\bs{u}}\right) \!-\! \nabla\pi,\\
\nabla \cdot \hat {\bs u} &= 0\label{eq:p-stokes-tangent2}
\end{alignat}
\end{subequations}
where $\bs{W}_{\!\varepsilon} := \bs{D}\bs{u}/|\bs{D}\bs{u}|_\varepsilon$ and ${\bs{W}}_{\!\bs\Lambda} := \bs{\Lambda}/|\bs{D}\bs{u}|_\varepsilon^{p-1}$, and where a rank-one fourth-order tensor acts on a second-order tensor as $(\bs A\otimes\bs B):\bs C := \bs A\,(\bs B\!:\!\bs C)$.
Note the direct analogy
with~\eqref{eq:lifted-linear-sys}: for $2\nu_0 = 1$, the terms $\nabla u$, $\bs w_\varepsilon$ and $\bs w_{\bs\lambda}$ in the $p$-Laplace problem play the role of $\bs{D}\bs{u}$, $\bs{W}_{\!\varepsilon}$ and $\bs{W}_{\!\bs\Lambda}$, respectively. The matrix--vector product $(\bs w_{\bs\lambda}\bs w_\varepsilon^\top)\nabla\hat u$ is replaced by its tensorial counterpart.

Writing
\begin{equation}\label{eq:p-stokes-Mk}
  \mathbb{M}:=|\bs{D}\bs{u}|_\varepsilon^{p-2}\Bigl(2\nu_0\,\mathbb{I}
  +(p-2)\,\sym\bigl(\bs{W}_{\!\bs\Lambda}\otimes\bs{W}_{\!\varepsilon}\bigr)\Bigr)
\end{equation}
for the resulting fourth-order diffusion tensor, where $\mathbb{I}$ is the identity on symmetric second-order tensors, the increment of the lifting variable is recovered pointwise from
\begin{equation}\label{eq:p-stokes-Lambda-hat}
  \bs{\hat\Lambda}:=\mathbb{M}:\bs{D}\bs{\hat u}-\bs{\Lambda}+\bs{S}_\varepsilon(\bs{D}\bs{u}),
\end{equation}
in analogy with \eqref{eq:lambda-hat}. Substituting \eqref{eq:p-stokes-Lambda-hat} into the linearized momentum equation cancels $\bs{\Lambda}$ from the right-hand side, which is why only $\bs{S}_\varepsilon(\bs{D}\bs{u})$ appears in \eqref{eq:p-stokes-tangent}; as in \cref{ssec:lifted-algorithm}, the lifting variable influences the velocity increment only through $\mathbb{M}$.

The two modifications of \cref{ssec:modifications} that guarantee the well-posedness of \eqref{eq:p-stokes-tangent} also carry over. The symmetrization $\sym(\bs{W}_{\!\bs\Lambda}\otimes\bs{W}_{\!\varepsilon}):=\tfrac12(\bs{W}_{\!\bs\Lambda}\otimes\bs{W}_{\!\varepsilon}+\bs{W}_{\!\varepsilon}\otimes\bs{W}_{\!\bs\Lambda})$ used in \eqref{eq:p-stokes-Mk} leaves the quadratic form $\bs{E}\!:\!\mathbb{M}\!:\!\bs{E}$ unchanged, but makes $\mathbb{M}$ self-adjoint, so that the velocity block of the saddle-point system \eqref{eq:p-stokes-tangent} is symmetric. The feasibility bound is imposed on $\bs{W}_{\!\bs\Lambda}$ as in \eqref{eq:safeguard}, but scaled by the viscosity constant: with $2-p<\beta<1$ as in \eqref{eq:beta-window}, we enforce
\begin{equation*}
  |\bs{W}_{\!\bs\Lambda}(\bs x)|\le\tfrac{2\nu_0\,\beta}{2-p}
  \qquad\text{for }\bs x\in\Omega
\end{equation*}
by replacing $\bs{\Lambda}$ in \eqref{eq:p-stokes-Mk} by its pointwise projection onto the ball of radius $\rho(\bs x):=\tfrac{2\nu_0\beta}{2-p}|\bs{D}\bs{u}(\bs x)|_\varepsilon^{p-1}$, as in \eqref{eq:discrete-lambda-projection}. The iterate $\bs{\Lambda}$ itself remains unchanged.
Since $|\bs{W}_{\!\varepsilon}|\le1$, this yields the  bounds
\begin{equation*}
  2\nu_0(1-\beta)\,|\bs{D}\bs{u}|_\varepsilon^{p-2}|\bs{E}|^2
  \;\le\;\bs{E}\!:\!\mathbb{M}\!:\!\bs{E}\;\le\;
  2\nu_0(1+\beta)\,|\bs{D}\bs{u}|_\varepsilon^{p-2}|\bs{E}|^2
\end{equation*}
for every symmetric $\bs E$, the direct counterpart of \eqref{eq:tensor-sandwich}.
The velocity block is therefore uniformly elliptic for every choice of $\bs{\Lambda}$, and standard Stokes solvers and preconditioners can be applied.

\section{Numerical results: $p$-Stokes}\label{sec:p-Stokes-numerics}

We discretize the incompressible Stokes equations using Taylor--Hood finite elements, i.e., we represent the velocity as a continuous piecewise quadratic vector field and the pressure as a continuous linear scalar field.
In this section, we present results for a stationary lid-driven cavity $p$-Stokes problem in \cref{ssec:lid-driven} and for %
a time-dependent problem modeling the $p$-Navier--Stokes flow around a cylinder in \cref{ssec:p-NS}. For brevity, we limit our comparison to the standard Newton (SN) and lifted Newton (LN) methods and examine how these methods behave under mesh refinement, as well as in a time-dependent setting where good initializations for the iterative solvers are available at each time step.

Our implementation again uses the Firedrake finite element library and is available from the same repository as the $p$-Laplace implementation. Analogously to before, the lifting variable $\bs\Lambda$ is not implemented as a finite element function. Instead, it is treated purely algebraically at the quadrature points. As before, we use the numerical parameters $\varepsilon=10^{-6}$, $\beta=1-10^{-4}$, $\gamma=0$. 
In each iteration of both methods, the resulting linear Stokes subproblems are again solved with a direct solver. Because these are standard symmetric saddle-point systems stemming from a stable Stokes finite element discretization, they are also amenable to existing iterative saddle-point preconditioners and solvers.

\subsection{Lid-driven cavity flow}\label{ssec:lid-driven}
First, we consider an incompressible lid-driven cavity $p$-Stokes flow problem. We consider \eqref{eq:p-stokes-strong} on $\Omega = (0,1)^2$, with $\bs f=0$ and $\nu_0=1/2$, and with the (smoothed) driving lid Dirichlet data $\bs{u}_D = [\frac{1}{2}(1 - \cos(2\pi x)), 0]^\top$ on the top boundary and the Dirichlet condition $\bs{u}_D = \bs{0}$ on the remaining boundaries. In these computations, we set the solution to the lid-driven cavity problem with $p=2$ as our initial guess in both the standard and lifted Newton methods. This avoids the failure of the nonlinear solvers for small $p$ values.

The flow field and the solver performance for different meshes and values of $p$ are shown in \cref{fig:lid-driven}. First, we observe that the driven flow close to the top boundary localizes in a very narrow boundary layer for $p=1.01$. Furthermore, as the table indicates, the lifted method requires only moderately more iterations as $p$ decreases, and its iteration counts are largely insensitive to the mesh: for $p=1.5$ it converges in four or five iterations on all three meshes, whereas SN grows from $19$ to $47$ iterations under the same refinement. The number of SN iterations rises steeply as $p$ approaches $1$; on the finest mesh, it requires $285$ iterations for $p=1.01$, compared with $22$ for LN, a factor of more than $13$. Both methods show a mild mesh dependence at $p=1.01$, but at levels that differ by more than an order of magnitude.

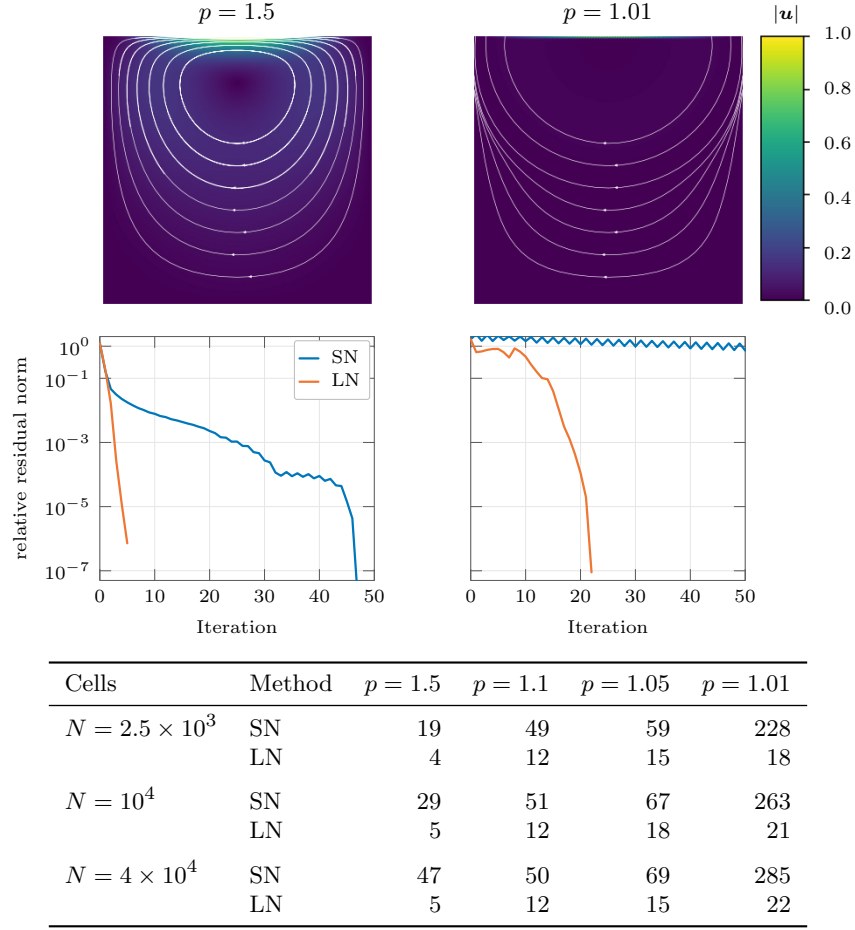
\begin{figure}[t]
    \centering
    \begin{tikzpicture}
    \def\LidPanelWidth{0.242\textwidth}
    \definecolor{lidConvBlue}{HTML}{0077BB}
    \definecolor{lidConvOrange}{HTML}{EE7733}
    \pgfplotstableread{data/lid_iteration_p150.txt}\lidIterationPOneFive
    \pgfplotstableread{data/lid_iteration_p101.txt}\lidIterationPOneOhOne
    \pgfplotsset{
        lid convergence axis/.style={
            width=\LidPanelWidth,
            height=0.215\textwidth,
            scale only axis,
            axis line style={black!65, line width=0.35pt},
            tick style={black!65, line width=0.35pt},
            grid=major,
            major grid style={black!10, line width=0.25pt},
            xlabel={Iteration},
            label style={font=\footnotesize},
            tick label style={font=\footnotesize},
            scaled ticks=false,
            unbounded coords=jump,
            filter discard warning=false,
            ymode=log,
            ymin=5e-8, ymax=2,
            ytick={1e-7,1e-5,1e-3,1e-1,1e0},
        },
        lid convergence curve/.style={
            solid,
            no marks,
            line width=0.85pt,
            line cap=round,
            line join=round,
        },
        lid standard curve/.style={
            lid convergence curve,
            color=lidConvBlue,
        },
        lid lifted curve/.style={
            lid convergence curve,
            color=lidConvOrange,
        },
        lid convergence legend/.style={
            draw=black!25,
            fill=white,
            rounded corners=1pt,
            font=\footnotesize,
            cells={anchor=west},
            legend image code/.code={
                \draw[
                    /pgfplots/mesh=false,
                    bar width=3pt,
                    bar shift=0pt,
                    mark repeat=2,
                    mark phase=2,
                    ##1
                ] plot coordinates {
                    (0cm,0cm) (0.125cm,0cm) (0.25cm,0cm)
                };
            },
            column sep=0.25em,
            row sep=-1.5pt,
            inner xsep=1.5pt,
            inner ysep=1.5pt,
        },
    }
        \begin{groupplot}[
            group style={
                group size=2 by 1,
                horizontal sep=0.085\textwidth,
            },
            lid convergence axis,
        ]
        \nextgroupplot[
            ylabel={relative residual norm},
            xmin=0, xmax=50,
            xtick={0,10,20,30,40,50},
            legend columns=1,
            legend style={
                lid convergence legend,
                at={(0.985,0.985)},
                anchor=north east,
            },
        ]
        \addplot[lid standard curve]
            table[x=iteration, y=PStokes] {\lidIterationPOneFive};
        \addlegendentry{SN}
        \addplot[lid lifted curve]
            table[x=iteration, y=PStokesLifted] {\lidIterationPOneFive};
        \addlegendentry{LN}

        \nextgroupplot[
            xmin=0, xmax=50,
            xtick={0,10,20,30,40,50},
            yticklabels=\empty,
        ]
        \addplot[lid standard curve]
            table[x=iteration, y=PStokes] {\lidIterationPOneOhOne};
        \addplot[lid lifted curve]
            table[x=iteration, y=PStokesLifted] {\lidIterationPOneOhOne};
        \end{groupplot}

        \node[anchor=south, inner sep=0pt] (lidSol150)
            at ([yshift=0.025\textwidth]group c1r1.north)
            {\includegraphics[
                width=\LidPanelWidth,
                trim=10.5bp 10.5bp 10.5bp 10.5bp,
                clip
            ]{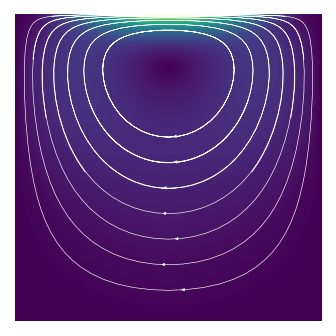}};
        \node[anchor=south, inner sep=0pt] (lidSol101)
            at ([yshift=0.025\textwidth]group c2r1.north)
            {\includegraphics[
                width=\LidPanelWidth,
                trim=10.5bp 10.5bp 10.5bp 10.5bp,
                clip
            ]{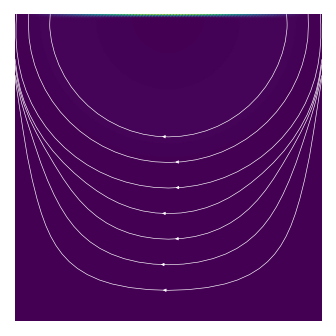}};
        \node[anchor=south, font=\small, inner sep=0pt]
            at ([yshift=3pt]lidSol150.north) {$p=1.5$};
        \node[anchor=south, font=\small, inner sep=0pt]
            at ([yshift=3pt]lidSol101.north) {$p=1.01$};

        \node[anchor=west, inner sep=0pt] (lidColorbar)
            at ([xshift=0.010\textwidth]lidSol101.east)
            {\includegraphics[
                height=\LidPanelWidth,
                trim=10.5bp 14.25bp 37.5bp 14.25bp,
                clip
            ]
                {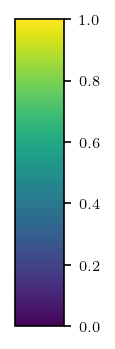}};
        \foreach \position/\value in {
            0/{$0.0$},.2/{$0.2$},.4/{$0.4$},
            .6/{$0.6$},.8/{$0.8$},1/{$1.0$}
        }{
            \node[anchor=west, font=\footnotesize, inner sep=0pt]
                at ([xshift=3pt]$(lidColorbar.south east)!\position!
                    (lidColorbar.north east)$) {\value};
        }
        \node[anchor=south, font=\footnotesize, inner sep=0pt]
            at ([yshift=3pt]lidColorbar.north)
            {$|\bs u|$};
    \end{tikzpicture}

    \par\vspace{1.5ex}
    {\small
    \begin{tabular}{llrrrr}
    \toprule
    Cells & Method & $p=1.5$ & $p=1.1$ & $p=1.05$ & $p=1.01$ \\
    \midrule
    $N=2.5\times10^3$ & SN & 19 & 49 & 59 & 228 \\
                      & LN & 4 & 12 & 15 & 18 \\
    \addlinespace
    $N=10^4$          & SN & 29 & 51 & 67 & 263 \\
                      & LN & 5 & 12 & 18 & 21 \\
    \addlinespace
    $N=4\times10^4$   & SN & 47 & 50 & 69 & 285 \\
                      & LN & 5 & 12 & 15 & 22 \\
    \bottomrule
    \end{tabular}}
    \caption{Solutions and convergence behavior for the lid-driven cavity $p$-Stokes problem. We consider different values of $p$ in $(1,2)$ and meshes with increasing numbers of cells $N$. The top row shows the velocity magnitude and streamlines for $p=1.5$ and $p=1.01$ on the most refined mesh, with a shared color scale for $|\bs u|$. The bottom row shows the relative $H^{-1}$ norm of the equation residual for SN and LN on the same mesh, using a shared vertical scale. The table reports the number of iterations for all values of $p$ and various mesh sizes.}
    \label{fig:lid-driven}
\end{figure}

\subsection{Flow around a cylinder for the $p$-Navier--Stokes equations} \label{ssec:p-NS}
As a second numerical example, we solve the time-dependent $p$-Navier--Stokes (p-NS) equations:
\begin{align*}
    \frac{\partial \bs{u}}{\partial t} + (\bs{u}\cdot\nabla)\bs{u}-\nabla\!\cdot\!\bs{S}\!\left(\bs{D}\bs{u}\right) + \nabla\pi
   = \bs{f} \quad \text{ on } \Omega,
\end{align*}
together with the divergence-free condition $\nabla\cdot \bs u=0$. Unlike the steady $p$-Stokes problem, the $p$-NS equations are time-dependent, and the convection term  $(\bs{u}\cdot\nabla)\bs{u}$ introduces an additional non-linearity.
For the $p$-NS equation, we build a LN method exactly as for the $p$-Stokes equation, i.e., by introducing a lifting variable that represents the deviatoric stress tensor $\bs{S}(\bs{D}\bs u)$. We discretize in time using the implicit Euler scheme. 
Following \cite{MM-pNavierStokes}, we compute the two-dimensional flow around a cylinder, see the left panels in \cref{fig:pstokes}, for $p=1.25, 1.1, 1.01$. The domain is of length $L = 2.2$ and height $H = 0.41$, with the cylinder of radius $0.05$ centered at $(0.2,0.2)$. We enforce the velocity boundary conditions 
\begin{align*}
    \bs{u}(x,y) = \omega(t) \left( \begin{array}{c}
         1 - \left( 1 - \frac{2y}{H}\right)^2 \\
          0
    \end{array}\right)
\end{align*}
on the left boundary, with $\omega(t) = \frac{1}{2}(1 - \cos{(\pi t)})$ for $t<1$ and $\omega(t) = 1$ for $t>1$. On the top and lower boundaries, we set $\bs{u} = 0$, and on the right boundary, $\bs{S}\!\left(\bs{D}\bs u\right)\bs{n} - \pi\,\bs{n} = 0$. As initial condition, we set the fluid velocity to zero.
We run the simulation up to time $t = 8$ for $3{,}392$ time steps on a mesh of $4{,}961$ cells. To be able to observe vortex shedding, we set $\nu_0 = 0.001$, which results, for $p=2$, in a nominal Reynolds number of about $100$.
These computations are challenging for two reasons: (1) For values of $p$ close to 1, enforcing the incoming velocity boundary condition (a parabolic profile) makes the problem stiff. 
(2) The nonlinear convection term is a potential source of numerical instabilities. We solve these issues as done in \cite{MM-pNavierStokes}: for (1), we enforce boundary conditions using Nitsche's method.
For (2), we stabilize our discretization with the so-called continuous interior penalty (CIP) method \cite{burman2007}. This method penalizes discontinuities in the velocity gradients.

Information on the iterations at each time step can be found in  \cref{fig:pstokes}. For these computations, convergence of the Newton methods is measured in terms of the relative $\ell^2$ norm of the discrete residual. For both SN and LN, the time steps that require the largest number of Newton iterations are those for $t < 4$, when the solution experiences the biggest changes at each time step. In these initial time steps, the number of iterations required by SN increases substantially as $p$ approaches $1$. Conversely, for LN, this increase is mild. For $t > 4$, both Newton methods require a similar number of iterations.

The advantage of LN is therefore concentrated in the stiff regime. At $p=1.25$, the two methods are essentially indistinguishable in cost, whereas at $p=1.01$, the mean drops from $9.86$ to $4.12$. Accumulated over the $3{,}392$ time steps, this amounts to roughly $33{,}400$ linear solves for SN compared to $14{,}000$ for LN, a reduction by a factor of about $2.4$. The peak cost per time step is affected far more strongly. At $p=1.01$, the maximum number of iterations in a single step falls from $140$ to $24$, which makes the cost of the lifted iteration substantially more predictable in a time-stepping context.

\begin{figure}[t]
  \centering
  \includegraphics[width=0.98\textwidth]{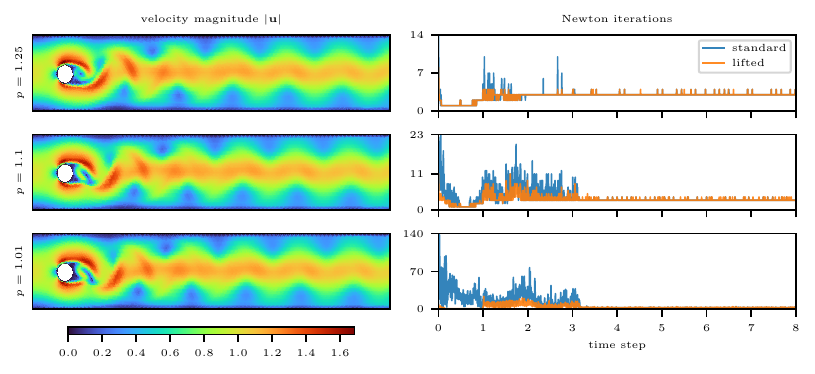}
    {\small
    \begin{tabular}{lrrr}
    \toprule
    Method & $p=1.25$ & $p=1.1$ & $p=1.01$ \\
    \midrule
    \makebox[2em][l]{SN}max  & 14   & 23   & 140 \\
    \makebox[2em][l]{}mean   & 2.78 & 3.57 & 9.86 \\
    \addlinespace
    \makebox[2em][l]{LN}max  & 8    & 11   & 24 \\
    \makebox[2em][l]{}mean   & 2.77 & 3.08 & 4.12 \\
    \bottomrule
    \end{tabular}}
  \caption{
  In the top left figures above, each row presents the velocity magnitude field $|\bs{u}|$ at $t = 8$. The top right figure shows the number of iterations per time step for three different $p$ values. In the table, for each $p$ we indicate the maximum and mean number of Newton iterations in time for SN and LN.}
  \label{fig:pstokes}
\end{figure}

\section{Discussion and extension}
We finish with a discussion of the main findings and possible extensions.

\noindent
\emph{Continuation.}
The standard remedy for the $p$-dependence of Newton's and Picard's
methods is continuation in $p$ \cite{HuangLiLiu07}, and sometimes in
$\varepsilon$. One solves a sequence of auxiliary problems, starting from
$p=2$ (or a large $\varepsilon$) and moving the parameters 
towards the target, using each solution to initialize the next problem.
This is a globalization by proximity to the solution, and its cost is
the sum of all intermediate solves.  None of the results presented here use continuation;
since the lifted iteration counts are largely insensitive to $p$ and $h$, there is no need for it.

\noindent
\emph{Krylov solvers.}
All linear systems reported here are solved directly.  For larger
problems in three dimensions, the structure of the
lifted tangent is favorable for iterative solvers.  By
\eqref{eq:tensor-sandwich}, the tensor $M_k$ is spectrally equivalent
to the Picard operator $|\nabla u_k|_\varepsilon^{p-2}I$, with
constants $1-\beta$ and $1+\beta$ that are independent of the mesh and of
$\varepsilon$. Any preconditioner that is spectrally equivalent to the discretized Picard operator is therefore a good candidate.
Following standard Newton--CG arguments, if the linear system in step (ii) of
\cref{ssec:lifted-algorithm} is solved by CG initialized by zero, every
CG iterate $\tilde u$ is a descent direction, and the Armijo
search \eqref{eq:armijo} remains well defined.

\noindent
\emph{Symmetrized versus non-symmetrized lifted Newton operator.}
The symmetrization \eqref{eq:lifted-tensor-sym} of the rank-one term
was introduced in \cref{ssec:modifications} for convenience only.  Since
$\bs z^\top B\bs z=\bs z^\top\sym(B)\bs z$ for every $B\in\mathbb
R^{d\times d}$, the quadratic form, along with the coercivity bounds
\eqref{eq:tensor-sandwich}, the descent property of the increment, and
the global convergence proof are identical for the symmetrized and the
nonsymmetric tensor $I+(p-2)\bs w_{\bs\lambda}\bs w_\varepsilon^\top$.
We have
compared both variants on the examples of \cref{sec:numerics-p-laplace}
and found the iteration histories to be nearly indistinguishable for most cases with moderate $p$, with
iteration counts that differ by at most one or two. More substantial differences are seen as $p\to 1$. For example, for the two-dimensional $p$-Laplace problem with $p=1.01$ and $N=25$, we obtain $23$ iterations with and $40$ iterations without symmetrization. The symmetrized
form is preferable in practice for other reasons as well. It allows us to use symmetric direct and iterative solvers,
and it is the form for which the descent
property of inexact CG steps discussed above holds.

\noindent
\emph{The case $p>2$.}
For $p>2$ the coefficient $|\nabla u|^{p-2}$ vanishes rather than blows up where $\nabla u=0$, so
no regularization is needed to define the residual, and the standard
Newton tensor
is uniformly positive definite.  The
difficulties for large $p$ are of a different kind: the diffusion
coefficient degenerates in flat regions, its dynamic range grows like
$|\nabla u|^{p-2}$ and leads to ill-conditioned systems and round-off
sensitivity, and the energy $J$ becomes increasingly non-quadratic in
steep regions.
The transformation \eqref{eq:extrem2-mod} is also the wrong one for $p>2$:
multiplying the constitutive relation by $|\nabla u|^{2-p}$ introduces
a singularity at $\nabla u=0$ where the original relation is perfectly
smooth.  The natural counterpart exchanges the roles of the primal and
dual variables and lifts the dual constitutive relation
$\nabla u=|\bs\lambda|^{q-2}\bs\lambda$ with $q=p/(p-1)\in(1,2)$, i.e.,
one applies the approach of this paper to the dual problem of
\cref{sec:duality}, in which the exponent $q$ now lies in the range
where the transformation is beneficial.
A systematic study of
lifting for $p>2$ is left for
future work.

\noindent
\emph{Open questions.}
Since $\beta>2-p$, the coercivity constant guaranteed by
\eqref{eq:tensor-sandwich} satisfies $1-\beta<p-1$, so our bounds
degrade as $p\to 1$ at the same rate as those for standard Newton.  The
observed iteration counts do not, and a local convergence analysis
that quantifies this gap, ideally with constants independent of $p$, remains open.  Equally open is a convergence result in the
infinite-dimensional setting or with $\varepsilon\to0$; however, this is a challenge that the lifted method shares with the standard variant.

\bmhead{Acknowledgements}
This work was supported in part by the US National Science Foundation (NSF) under awards \#2343866, \#2411229 and \#2411349.

\nocite{snbibopts}
\bibliography{p_laplace}

\appendix

\section{Alternative projection of the lifting variable}
\label{app:project}
In the method of \cref{ssec:lifted-algorithm}, the projection \eqref{eq:discrete-lambda-projection} is applied only where 
the lifting variable enters the diffusion tensor, while the iterate $\bs\lambda_k$ itself is evolved unconstrained. An alternative is to project the iterate as well, i.e., to replace step~(v) by
\begin{equation}\label{eq:lambda-projected-iterate}
  u_{k+1}:=u_k+\alpha_k\hat u_k,
  \qquad
  \widetilde{\bs\lambda}_{k+1}:=\bs\lambda_k+\alpha_k\hat{\bs\lambda}_k,
  \qquad
  \bs\lambda_{k+1}:=\frac{\widetilde{\bs\lambda}_{k+1}}{\max\bigl(1,\,|\widetilde{\bs\lambda}_{k+1}|/\rho_{k+1}\bigr)},
\end{equation}
with $\rho_{k+1}$ as in \eqref{eq:discrete-lambda-projection} evaluated at $u_{k+1}$, and to use $\bs\lambda_k$ itself in \eqref{eq:Mk}. As in \eqref{eq:discrete-lambda-projection}, the last relation is understood pointwise 
and gives $|\bs\lambda_{k+1}|=\min\bigl(|\widetilde{\bs\lambda}_{k+1}|,\rho_{k+1}\bigr)$. Since the feasibility bound then holds for the iterate itself, the projection in the definition of $M_k$ becomes unnecessary.

The two variants are not equivalent in the initial phase, but the analysis of \cref{ssec:p-Lapl-convergence} covers both. The primal statements of \cref{thm:finite-dimensional-lifted-global-convergence} use the projection 
only through the spectral bounds \eqref{eq:spectral-lower-upper}, which hold for either variant. For the flux convergence of corollary~\ref{cor:quadrature-flux-convergence} one notes in addition that the projection in \eqref{eq:lambda-projected-iterate} leaves $S_{k+1}:=|\nabla u_{k+1}|_\varepsilon^{p-2}\nabla u_{k+1}$ untouched, since $\beta>2-p$ implies
\begin{equation*}
  |S_{k+1}|
  =|\nabla u_{k+1}|_\varepsilon^{p-2}\,|\nabla u_{k+1}|
  \le |\nabla u_{k+1}|_\varepsilon^{p-1}
  \le\frac{\beta}{2-p}|\nabla u_{k+1}|_\varepsilon^{p-1}
  =\rho_{k+1}
\end{equation*}
pointwise. Since the scaling in \eqref{eq:lambda-projected-iterate} is
the pointwise projection onto the ball of radius $\rho_{k+1}$ and
hence nonexpansive, the recursion \eqref{eq:Ek-recursion} is preserved, and the conclusion is unchanged.  In practice, the two methods yield
almost identical behavior in terms of iteration numbers, as shown in \cref{tab:p-Lapl-2d-appendix}.

\begin{table}[ht]
    \centering
    \caption{Comparison of iterations required by lifted Newton method from \cref{ssec:lifted-algorithm} and the variant discussed here for the two-dimensional $p$-Laplace example of \cref{ssec:plaplace-2d}, using a mesh with $10{,}000$ triangles.}  \label{tab:p-Lapl-2d-appendix}
    \begin{tabular}{lrrrrr}
        \toprule
        Method & $p=1.6$ & $p=1.3$ & $p=1.1$ & $p=1.05$ & $p=1.01$ \\
        \midrule
        Lifted Newton from \cref{ssec:lifted-algorithm} & 11 & 10 & 15 & 19 & 21 \\
        Lifted Newton variant & 11 & 10 & 15 & 18 & 21 \\
        \bottomrule
    \end{tabular}
\end{table}

\section{One-dimensional finite difference example}
\label{app:one_dim_flux}

In \cref{ssec:robustness}, we give a qualitative explanation of the mechanisms behind the improved convergence of the lifted Newton method for $p$ close to one. To make this explanation more concrete, this appendix analyzes the standard and lifted Newton steps for a one-dimensional $p$-Laplace problem discretized with finite differences. The key observation is that both methods compute the \emph{same} increment of the flux $|\nabla u|_{\varepsilon}^{p-2}\nabla u$, and differ only in how this flux increment is converted into an increment of the gradient: standard Newton uses the tangent slope of the constitutive function, which is nearly flat for $p\approx 1$, and hence overshoots by a factor of $1/(p-1)$, whereas lifted Newton uses a secant slope and gets close to the solution quickly.

Consider the $p$-Laplace equation \eqref{eq:p-laplace} on $\Omega=(0,1)$ with homogeneous Dirichlet boundary conditions. We partition $\Omega$ using $n+2$ uniformly distributed nodes with spacing $h=1/(n+1)$ and denote by $\bs u=(u_1,\dots,u_n)^\top$ the interior degrees of freedom, with $u_0=u_{n+1}=0$. The finite difference gradient is $\bs y=B\bs u\in\mathbb R^{n+1}$, $y_i=(u_i-u_{i-1})/h$ for $i=1,\dots,n+1$. The discrete problem is the minimization of
\begin{equation*}
  J(\bs u):=\frac1p\sum_{i=1}^{n+1}|y_i(\bs u)|_\varepsilon^p-\bs f^{\!\top}\bs u,\qquad |y|_\varepsilon:=(y^2+\varepsilon^2)^{1/2}.
\end{equation*}
Let $g(t):=|t|_\varepsilon^{p-2}t$ denote the scalar constitutive function (the flux as a function of the gradient) and $\mathbf g(\bs y)$ the vector with components $g(y_i)$. The residual is $\bs r(\bs u)=B^\top\mathbf g(B\bs u)-\bs f$, with components
  $r_i=(g(y_i)-g(y_{i+1}))/{h}-f_i$.%

To obtain a problem with a plateau in the solution, we choose $\bs f$ such that the discrete solution is $\bs u^*=\bs 1_n:=(1,\dots,1)^\top$, i.e., $y_1^*=1/h$, $y_{n+1}^*=-1/h$ and $y_i^*=0$ for $i=2,\dots,n$. From $\bs r(\bs u^*)=0$ and $g(0)=0$ we obtain
  $f_1=f_n={g(1/h)}/{h}$ and $f_i=0$  for $i=2,\dots,n-1$,
so that the solution is flat in the sense of \cref{ssec:robustness} on all but the two boundary intervals.

Both Newton methods compute an increment $\hat{\bs u}$ from a linear system of the form
\begin{equation}\label{eq:appC-newton}
  B^\top D\,B\,\hat{\bs u}=-\bs r(\bs u),\qquad D=\operatorname{diag}(d_1,\dots,d_{n+1})\in\mathbb R^{(n+1)\times(n+1)},
\end{equation}
and differ only in the diagonal coefficients. For standard Newton, $B^\top D^S B$ is the Hessian of $J$, with
\begin{equation*}
  d_i^S:=g'(y_i)=|y_i|_\varepsilon^{p-2}\Bigl[1+(p-2)\frac{y_i^2}{|y_i|_\varepsilon^2}\Bigr].
\end{equation*}
For lifted Newton with lifting variable $\bs\lambda\in\mathbb R^{n+1}$, the one-dimensional counterpart of the diffusion tensor \eqref{eq:lifted-tensor-sym} is
\begin{equation*}
  d_i^L:=|y_i|_\varepsilon^{p-2}\Bigl[1+(p-2)\frac{\pi_u(\lambda_i)\,y_i}{|y_i|_\varepsilon^{p}}\Bigr],
\end{equation*}
where $\pi_u(\lambda_i)$ denotes the feasibility projection of $\lambda_i$, cf.~\eqref{eq:safeguard}. Note that $d_i^L=d_i^S$ if $\lambda_i=g(y_i)$, in accordance with \cref{ssec:newton}. The lifting variable is updated according to step (iii) of \cref{ssec:lifted-algorithm}, which, for a unit step, reads
\begin{equation}\label{eq:appC-lambda-update}
  \bs\lambda_{k+1}=\mathbf g(\bs y_k)+D^L\,\hat{\bs y}_k,\qquad \hat{\bs y}_k:=B\,\hat{\bs u}_k,
\end{equation}
and we initialize with $\bs\lambda_0=\bs 0$.

\subsection{Expansions for small values of $p-1$}
We introduce $\tau:=p-1$ and assume $\tau \ll 1$. Although $\varepsilon$ determines which regions are flat, we assume for simplicity that $|y_i|\gg\varepsilon$ for all $i$, so that $|y_i|_\varepsilon\approx|y_i|$; we found this to be accurate for the iterates in \cref{fig:fd_result}, where $\varepsilon=10^{-6}$. Since $|y|^{p-2}=|y|^{-1}|y|^{\tau}=|y|^{-1}(1+O(\tau))$ for $|y|$ bounded away from zero and infinity, to leading order in $\tau$,
\begin{equation}\label{eq:appC-expansions}
  g(y)\approx\operatorname{sign}(y),\qquad
  d^S\approx\frac{\tau}{|y|},\qquad
  d^L\approx\frac{1-\pi_u(\lambda)\operatorname{sign}(y)}{|y|},
\end{equation}
where we used $p-2=\tau-1$. The constitutive function is thus nearly the sign function: its tangent slope $d^S$ is of order $\tau$, while the secant slope $g(y)/y\approx1/|y|$ through the origin is of order one. Lifted Newton with $\pi_u(\lambda)=O(\tau)$ uses this secant slope; with $\lambda=g(y)$ it recovers the tangent slope.

We introduce the linearized flux increment $\hat{\bs z}:=D\,B\,\hat{\bs u}\in\mathbb R^{n+1}$, so that
  $\hat{y}_i=\hat z_i/d_i$ for $i=1,\dots,n+1$.
In terms of $\hat{\bs z}$, the Newton system \eqref{eq:appC-newton} reads $B^\top\hat{\bs z}=-\bs r$, i.e., $(\hat z_i-\hat z_{i+1})/h=-r_i$ for $i=1,\dots,n$, and therefore
\begin{equation}\label{eq:appC-flux-increment}
  \hat z_i=\hat z_1+h\rho_i,\qquad\rho_i:=\sum_{j=1}^{i-1}r_j,\qquad i=1,\dots,n+1 .
\end{equation}
In one dimension, the flux increment is thus determined by the residual up to a single constant $\hat z_1$. This constant is fixed by the homogeneous boundary conditions, which imply that $\hat u_{n+1}-\hat u_0=h\sum_{i=1}^{n+1}\hat{y}_i=0$, and therefore
\begin{equation}\label{eq:appC-z1}
  \hat z_1=-\sum_{i=1}^{n+1}w_ih\rho_i,\qquad w_i:=\frac{d_i^{-1}}{\sum_{\ell=1}^{n+1}d_\ell^{-1}},
\end{equation}
that is, $-\hat z_1$ is a convex combination of the values $h\rho_i$. Equations \eqref{eq:appC-flux-increment}--\eqref{eq:appC-z1} are therefore an explicit solution formula for \eqref{eq:appC-newton}, valid for any positive diagonal $D$.

\subsection{The flux increment is the same for both methods}
By the expressions for the residual $\bs r$ and the forcing $\bs f$, the sums in \eqref{eq:appC-flux-increment} telescope,
\begin{equation*}
  h\rho_i=g(y_1)-g(y_i)-hf_1=g(y_1)-g(1/h)-g(y_i),\qquad i=2,\dots,n,
\end{equation*}
and $h\rho_{n+1}=g(y_1)-g(y_{n+1})-2g(1/h)$. Close to the solution, $y_1>0$ and $y_{n+1}<0$, so that by \eqref{eq:appC-expansions} $g(y_1)-g(1/h)=O(\tau)$ and $g(y_{n+1})+g(1/h)=O(\tau)$. Thus,
\begin{equation}\label{eq:appC-rho}
  h\rho_1=0,\qquad h\rho_i\approx-g(y_i)\approx-\operatorname{sign}(y_i)\quad \text{for $i=2,\dots,n$},\qquad h\rho_{n+1}=O(\tau).
\end{equation}
To estimate the constant $\hat z_1$ in \eqref{eq:appC-z1}, let $\delta:=\max_i|u_i-1|$ denote the distance of the iterate from the solution and assume $\delta\ll1$. For both methods, $d_i^{-1}\propto|y_i|$ in the interior by \eqref{eq:appC-expansions}, while at the two boundary intervals $d_1^{-1}\approx d_{n+1}^{-1}\approx1/(\tau h)$ for standard Newton and for lifted Newton with $\lambda_1\approx g(y_1)$, $\lambda_{n+1}\approx g(y_{n+1})$, and $d_1^{-1}\approx d_{n+1}^{-1}\approx1/h$ for the first lifted Newton step. Using \eqref{eq:appC-rho} and $h\sum_{i=2}^n|y_i|\operatorname{sign}(y_i)=u_n-u_1$, the convex combination \eqref{eq:appC-z1} evaluates to
\begin{equation*}
\begin{aligned}
  \hat z_1&\approx\frac{u_n-u_1}{2+\mathrm{TV}(\bs u)}+O(\tau)&&\text{(standard Newton, first lifted Newton step)},\\
  \hat z_1&\approx\frac{\tau\,(u_n-u_1)}{2+\tau\,\mathrm{TV}(\bs u)}+O(\tau)&&\text{(later lifted Newton steps)},
\end{aligned}
\end{equation*}
with $\mathrm{TV}(\bs u):=h\sum_{i=2}^n|y_i|$. In all cases, $|\hat z_1|\lesssim\delta+O(\tau)$, which is small compared to the unit-size values in \eqref{eq:appC-rho}. Inserting into \eqref{eq:appC-flux-increment} yields, in the flat region,
\begin{equation}\label{eq:appC-z-flat}
  \hat z_i\approx-g(y_i)\approx-\operatorname{sign}(y_i),\qquad i=2,\dots,n,
\end{equation}
for both methods, i.e., they compute the same flux increment. As desired, the linearized flux $g(y_i)+\hat z_i$ after the step vanishes to leading order, which is the target value of the flux in the flat region.

\subsection{Comparison of the primal increments}
The two methods differ only in the conversion $\hat y_i = \hat z_i/d_i$ of the flux increment into a gradient increment. Combining \eqref{eq:appC-expansions} and \eqref{eq:appC-z-flat}, we find for $i=2,\dots,n$
\begin{equation*}
  \hat{y}_i^S\approx-\frac{|y_i|}{\tau}\operatorname{sign}(y_i)=-\frac{y_i}{\tau},
  \qquad
  \hat{y}_i^L\approx-\frac{|y_i|\operatorname{sign}(y_i)}{1-\pi_u(\lambda_i)\operatorname{sign}(y_i)}=-\frac{y_i}{1-\pi_u(\lambda_i)\operatorname{sign}(y_i)}.
\end{equation*}
With a unit step, standard Newton replaces $y_i$ by $y_i(1-1/\tau)$: every gradient in the flat region changes sign and grows in magnitude by a factor of approximately $1/\tau$. Since the residual $\bs r$ is dominated by $\operatorname{sign}(y_i)-\operatorname{sign}(y_{i+1})$, the residual changes sign as well, and the next step reverses the motion. This is the zig-zag behavior described in \cref{ssec:robustness} and visible in \cref{fig:fd_result}. A damping parameter $\alpha\approx\tau$ would be needed in the flat region to make the step useful, while the steep boundary intervals, where $g$ is well approximated by its tangent, allow for $\alpha\approx1$. The Armijo line search must compromise between the two, which explains the slow convergence of standard Newton on this problem. Lifted Newton with $\pi_u(\lambda_i)=O(\tau)$ instead gives $\hat{y}_i^L\approx-y_i$, so that a unit step sets the gradients in the flat region to zero up to $O(\tau)$.

\subsection{The lifting variable in the flat region}
It remains to justify that $\lambda_i=O(\tau)$ for $i=2,\dots,n$ and all $k\ge1$. By the definition of $\hat{\bs z}$, the update \eqref{eq:appC-lambda-update} of the lifting variable can be written as
\begin{equation*}
  \bs\lambda_{k+1}=\mathbf g(\bs y_k)+\hat{\bs z}_k,
\end{equation*}
that is, the new lifting variable is exactly the linearized flux after the step. Hence, by \eqref{eq:appC-expansions} and  \eqref{eq:appC-z-flat}, the components $\lambda_{k+1,i}$, $i=2,\dots,n$, of $\bs\lambda_{k+1}$ in the flat region are of order $\tau$.
This holds for all values of $\bs\lambda_k$, since \eqref{eq:appC-z-flat} does not depend on $D$. In particular, for the first step, $\bs\lambda_0=\bs 0$ gives $d_i^L=|y_i|_\varepsilon^{p-2}\approx1/|y_i|$, so that the first lifted Newton step coincides with a Picard step and equals $\tau$ times the first standard Newton step. After the step, $\lambda_i=O(\tau)$ in the flat region, \eqref{eq:appC-expansions} gives $d_i^L\approx1/|y_i|$ again, and the argument repeats. On the two boundary intervals, in contrast, $\lambda_1\approx g(y_1)\approx1$ and $\lambda_{n+1}\approx-1$ after the first step, so that there $d^L\approx d^S$ and lifted Newton reduces to standard Newton, which converges quadratically in the steep region.

To summarize, the two linearizations agree on what the flux should do. Standard Newton converts the flux increment into a primal increment using the tangent slope $\tau/|y|$ of the nearly flat function $g$ and hence overshoots by $1/\tau$, whereas lifted Newton uses the secant slope $1/|y|$ and lands on the flat solution, up to $O(\tau)$, in one step. This is the one-dimensional instance of the mechanism described in \cref{ssec:robustness}: the dual direction lags behind the primal direction and filters out the extreme sensitivity of $\bs w_\varepsilon$ in regions where the gradient is small.

\end{document}